\documentclass[10pt, oneside, a4paper]{article}
\usepackage{geometry}
\title{Elliptic curves with large Tate--Shafarevich groups over $\F_q(t)$} 
\author{%
 {\Large Richard \textsc{Griffon}\footnote{Correspondence to be sent to \href{mailto:richard.griffon@unibas.ch}{richard.griffon@unibas.ch}.}} 
\and 
{\Large Guus De \textsc{Wit}}}
\date{}

\usepackage[english]{babel}
\usepackage[T1]{fontenc}
\usepackage{lmodern}
\usepackage{amssymb,amsmath,amscd}  
\usepackage{mathrsfs}
\usepackage{multicol}
\usepackage{enumerate}
\usepackage{graphicx}
\usepackage[table, dvipsnames]{xcolor}	
 
\usepackage{amsthm}
{\theoremstyle{plain}%
\newtheorem{theo}{Theorem}[section]}
{\theoremstyle{plain}%
\newtheorem{coro}[theo]{Corollary}}
{\newtheorem{lemm}[theo]{Lemma}}
{\newtheorem{prop}[theo]{Proposition}}
{}
{\newtheorem{itheo}{Theorem}}
{\newtheorem{itheorem}{Theorem}}
{\newtheorem{iconj}[itheorem]{Conjecture}}

{\theoremstyle{definition}%
\newtheorem{defi}[theo]{Definition}}
{\theoremstyle{remark}%
}
{\theoremstyle{remark}%
\newtheorem{rema}[theo]{Remark}}

\usepackage{fancyhdr}
\usepackage{titlesec}
\titleformat{\subsection}[runin]%
        {\bfseries}% 
        {\thesubsection.}% 
        {0.2em}{}[.\hspace{0.3em}-- ]
        
 \titleformat{\section}%
        {\center\large\bfseries}%
        {\thesection.}%
        {0.2em}{}[]%

\usepackage{bm}
\DeclareMathOperator{\rk}{rank}
\DeclareMathOperator{\ord}{ord}
\newcommand{\val}{v}
\DeclareMathOperator{\Gal}{Gal}

\DeclareMathOperator{\trace}{Tr}
\newcommand{\norm}{\mathbf{N}}
\DeclareMathOperator{\DEGRE}{deg }
\renewcommand{\deg}{\DEGRE}

\newcommand{\R}{\ensuremath{\mathbb{R}}}
\newcommand{\Q}{\ensuremath{\mathbb{Q}}}
\newcommand{\C}{\ensuremath{\mathbb{C}}}
\newcommand{\Z}{\ensuremath{\mathbb{Z}}}
\renewcommand{\P}{\ensuremath{\mathbb{P}}}
\newcommand{\F}{\ensuremath{\mathbb{F}}}
\renewcommand{\H}{\ensuremath{\mathrm{H}}}
\newcommand{\Qbar}{\ensuremath{\bar{\mathbb{Q}}}}
\newcommand{\e}{\ensuremath{\mathrm{e}}}
\newcommand{\dd}{\ensuremath{\,\mathrm{d}}}
\renewcommand{\bar}[1]{\ensuremath{\overline{#1}}}
\newcommand{\tam}{\tau}
\newcommand{\cond}{\mathcal{N}}
\newcommand{\tors}{_{\mathrm{tors}}}

\renewcommand{\epsilon}{\varepsilon}
\newcommand{\ie}{\textit{i.e.}{}}

\input cyracc.def
\DeclareFontFamily{U}{russian}{}
\DeclareFontShape{U}{russian}{m}{n}
	{ <5><6> wncyr5
	<7><8><9> wncyr7
	<10><10.95><12><14.4><17.28><20.74><24.88> wncyr10 }{}
\DeclareSymbolFont{Russian}{U}{russian}{m}{n}
\DeclareSymbolFontAlphabet{\mathcyr}{Russian}
\makeatletter
\let\@math@cyr\mathcyr
\renewcommand{\mathcyr}[1]{\@math@cyr{\cyracc #1}}
\makeatother
\newcommand{\sha}{\mathcyr{SH}}

\newcommand{\gP}{\mathfrak{P}}
\newcommand{\BS}{\mathfrak{Bs}}
\newcommand{\into}{\hookrightarrow}
\newcommand{\gauss}[3]{\mathrm{G}_{#1}(#2,#3)}
\newcommand{\gaun}{\hm{\mathsf{g}}}
\newcommand{\gauan}{\bm{\varepsilon}}
\newcommand{\kloos}[3]{\mathrm{Kl}_{#1}(#2;#3)}
\newcommand{\klo}{\mathsf{kl}}
\newcommand{\Kloosn}{\hm{\mathsf{Kl}}}
\newcommand{\kln}{\hm{\mathsf{kl}}}
\newcommand{\kloan}{\bm{\theta}}

\usepackage{url}
\usepackage[colorlinks=true, citecolor= LimeGreen, linkcolor=  OliveGreen!90, urlcolor=black!60]{hyperref}

\newcommand{\refGauss}[1]{\hyperref[#1]{\textsf{(Ga\,\ref*{#1})}}}
\newcommand{\refKloos}[1]{\hyperref[#1]{\textsf{(Kl\,\ref*{#1})}}}
\renewcommand{\O}{\mathcal{O}}

\begin{document}
\pagestyle{fancy}

\maketitle 

\noindent\hfill\rule{7cm}{0.5pt}\hfill\phantom{.}

\paragraph{Abstract --}
Let $\mathbb{F}_q$ be a finite field of odd characteristic $p$. 
We exhibit elliptic curves over the rational function field $K = \mathbb{F}_q(t)$ whose Tate-Shafarevich groups are large.
More precisely, we consider certain infinite sequences of explicit elliptic curves $E$, for which we prove that their Tate--Shafarevich group~${\sha(E)}$ is finite and satisfies $|\sha(E)| = H(E)^{1+o(1)}$ as $H(E)\to\infty$,
where $H(E)$ denotes the exponential differential height of $E$. 
The elliptic curves in these sequences are pairwise neither isogenous nor geometrically isomorphic.
We further show that the $p$-primary part of their Tate--Shafarevich group is trivial.

The proof involves explicitly computing the $L$-functions of these elliptic curves, 
proving the BSD conjecture for them, 
and obtaining estimates on the size of the central value of their $L$-function.


\smallskip
\noindent {\it 2010 Math.\ Subj.\ Classification:}  
11G05, %Elliptic curves over global fields
11G40, %L-functions of varieties over global fields; BSD conjecture 
14G10, %Zeta-functions and related questions
11L05. %Gauss and Kloosterman sums; generalisations

\noindent\hfill\rule{7cm}{0.5pt}\hfill\phantom{.}

 %% -- BEGIN CONTENT  -- %% 
\section*{Introduction}\setcounter{section}{0}

\subsection{The size of Tate--Shafarevich groups of elliptic curves}
Fix a finite field  $\F_q$ of odd characteristic $p$ and let $K:=\F_q(t)$ denote the rational function field over~$\F_q$.
Let $E$ be an elliptic curve over~$K$ which, we assume, is non-isotrivial (\ie{}, its $j$-invariant $j(E)\in K$ does not lie in $\F_q$). 
One attaches to $E$ its Tate--Shafarevich group, denoted by $\sha(E)$.
Among other reasons, the arithmetic significance of $\sha(E)$ stems from its measuring, in a certain sense, how badly the local--global principle fails for $E$. 
The Tate--Shafarevich group remains a mysterious object;
for instance, the finiteness of $\sha(E)$ is still conjectural in general, even though it has been proved in a number of cases.
Let us assume for now that $\sha(E)$ is indeed finite: what can then be said about its size?
There are several natural choices of numerical invariants of $E$ to compare $|\sha(E)|$ to.
Here, we choose the exponential differential height $H(E)$ and the conductor $N(E)$: these are defined by
\[ H(E) := q^{\frac{1}{12}\deg\Delta_{\min}(E)} \qquad \text{ and }\qquad 
    N(E) := q^{\deg\cond(E)},\]
where $\Delta_{\min}(E)$ and  $\cond(E)$ denote the minimal discriminant and the conductor divisors of $E$, respectively.
Goldfeld and Szpiro \cite{GoldfeldSzpiro} have proven upper bounds on the order of $\sha(E)$ in terms of these two invariants.
Let us quote a special case of their result:
	\begin{itheorem}[Goldfeld--Szpiro \cite{GoldfeldSzpiro}]
	\label{itheo.GoSz}
	Let $E$ be a non-isotrivial elliptic curve over $K$. 
	Assume that $E$ has finite Tate--Shafarevich group. 
	\begin{enumerate}[(i)]
	\setlength{\itemsep}{0pt}
	\item 
	Then, for all $\epsilon>0$, one has $|\sha(E)|\ll_{q, \epsilon} H(E)^{1+\epsilon}$.
	\item 
	If, moreover, $j(E)\in K$ is not a $p$-th power in $K$, then, 
	
	for all $\epsilon>0$, one has $|\sha(E)|\ll_{q, \epsilon} N(E)^{1/2+\epsilon}$.
	\end{enumerate}
	\end{itheorem}

\noindent
In this statement, item {\it(ii)} follows from item {\it(i)} by applying Szpiro's inequality for elliptic curves with separable $j$-invariant.
It is then natural to wonder about the optimality of {\it(i)} and {\it(ii)}: apart from the $\epsilon$'s, are the exponents of the height ($1$) and of the conductor ($1/2$) in these upper bounds best possible? 
In other words, as $E$ ranges over all elliptic curves over $K$, what is the largest power of $H(E)$ -- or $N(E)$ -- that does appear in $|\sha(E)|$, up to a $\pm\epsilon$ for all $\epsilon>0$?

\medskip
In the analogous setting of elliptic curves over $\Q$, the analogue of Theorem \ref{itheo.GoSz} is known to follow from the ABC conjecture (see \S1--3 in \cite{GoldfeldSzpiro}).
Moreover, de Weger \cite{deWeger_ABC} conjectures that the exponents $1$ and $1/2$ should indeed be optimal in that context (see Conjectures 2 and 4 there).
Following the analogy between the arithmetics of elliptic curves over $\Q$ and over $K$, one can translate the statement of de Weger's conjecture  
(the quantities $H(E)$ and $N(E)$ here correspond to the $1/12$-th power of the minimal discriminant and to the conductor of an elliptic curve over $\Q$, respectively). 
This translation results in the following:
	\begin{iconj}[de Weger \cite{deWeger_ABC}]
	Assuming  finiteness of the relevant Tate--Shafarevich groups,
	\label{iconj.deW}
	\begin{enumerate}[(i)]
		\item\label{iconj.H} 
		For any $\epsilon>0$, there are infinitely many elliptic curves $E/K$ such that 
		$|\sha(E)|\gg_{q, \epsilon} H(E)^{1-\epsilon}$.
	
		\item\label{iconj.N} 
		For any $\epsilon>0$, there are infinitely many elliptic curves $E/K$ such that 
		$|\sha(E)|\gg_{q, \epsilon} N(E)^{1/2-\epsilon}$.
	\end{enumerate}	
	\end{iconj}

\noindent
In the same paper, de Weger proves, conditionally to the Birch and Swinnerton-Dyer (BSD) conjecture, that (the analogue of) Conjecture \ref{iconj.deW}\eqref{iconj.H} holds for elliptic curves over~$\Q$.
Prior to \cite{deWeger_ABC}, Mai and Murty \cite[Theorem 2]{MaiMurty_QuadTwists} had shown, again conditionally to the BSD conjecture,  a weaker version of  Conjecture \ref{iconj.deW}\eqref{iconj.N}  for elliptic curves over $\Q$, with the exponent $1/2$ replaced by $1/4$.
Both of these results rely on considering sequences of  well-chosen quadratic twists of a given elliptic curve.  

\medskip
%We now come back to the setting of elliptic curves over $K$. 
In the context of elliptic curves over $K$, the above-stated Conjecture~\ref{iconj.deW}\eqref{iconj.H}  is not difficult to prove.
One can indeed take advantage of the existence of inseparable isogenies of large degree to construct sequences of elliptic curves over $K$ with large Tate--Shafarevich groups (see \S\ref{sec.large.sha.Frob} where we build one such example). 
By construction of these sequences, however,  the elliptic curves therein are $K$-isogenous: one cannot, therefore,  hope to deduce from these any result towards Conjecture \ref{iconj.deW}\eqref{iconj.N}.
In these sequences, one also notices that the $p$-primary parts of the Tate--Shafarevich groups are ``large''; and that, actually, the order of the $p$-primary part already accounts for the observed ``large $\sha$'' phenomenon. 

We are therefore led to ask the following questions: {\it% 
For a given $\epsilon>0$, are there infinitely many \emph{pairwise non $K$-isogenous} elliptic curves~$E/K$ such that $|\sha(E)|\gg_{q, \epsilon} H(E)^{1-\epsilon}$?
If so, is the fact that their Tate--Shafarevich groups  are ``large''  always explained by a ``large $p$-primary part of $\sha$'' phenomenon?}

In this paper, we first give a positive answer to the first question: 
	\begin{itheo}\label{itheo.sha.H}
	For all $\epsilon>0$, there are infinitely many \emph{pairwise non $K$-isogenous} elliptic curves $E/K$ with finite Tate--Shafarevich group, 
	such that $|\sha(E)|\geq H(E)^{1-\epsilon}$.
	\end{itheo}

\noindent
We also prove a result in direction of Conjecture \ref{iconj.deW}\eqref{iconj.N}, with the same exponent $1/4$ as in \cite{MaiMurty_QuadTwists}:
	\begin{itheo}\label{itheo.sha.N}
	For all $\epsilon>0$, there are infinitely many \emph{pairwise non $\bar{K}$-isomorphic} elliptic curves $E/K$ with separable $j$-invariant and finite Tate--Shafarevich group,
	  such that $|\sha(E)|\geq N(E)^{1/4-\epsilon}$.	
	\end{itheo}

In contrast to the aforementioned results concerning elliptic curves over $\Q$, Theorems \ref{itheo.sha.H} and~\ref{itheo.sha.N} are unconditional, and the involved elliptic curves are not quadratic twists of each other. 
Our proof is constructive and effective: we exhibit sequence(s) of elliptic curves over $K$ satisfying these properties and we provide explicit bounds on the order of their Tate--Sharafarevich groups.
We also prove that the $p$-primary parts of the involved Tate--Shafarevich groups are trivial, thus answering negatively the second question in italics raised above.

\subsection{Elliptic curves with large Tate--Shafarevich groups}

Theorems \ref{itheo.sha.H} and \ref{itheo.sha.N} both follow from our main theorem, which we now state. 
For any parameter $\gamma\in\F_q^\times$ and any integer $a\geq 1$, we write $\wp_a(t):=t^{q^a}-t\in\F_q[t]$ and consider the elliptic curve $E_{\gamma, a}$ defined over $K$ by the Weierstrass model:
\begin{equation}\label{ieq.Wmod}
E_{\gamma, a}:\qquad y^2=x\cdot\big(x^2+\wp_a(t)\cdot x + \gamma\big). \\
\end{equation}

\noindent
The main result of this paper is the following:
	\begin{itheo}\label{itheo.large.sha}
	In the above setting, 
	\begin{enumerate}[(1)]
	\item 
	As $\gamma\in\F_q^\times$  and $a\geq 1$ vary, the curves $E_{\gamma, a}$ are pairwise neither $\bar{K}$-isomorphic nor $K$-isogenous.
	\item 
	For any $\gamma\in\F_q^\times$ and $a\geq 1$, the Tate--Shafarevich group $\sha(E_{\gamma, a})$ is finite.
	\item\label{item.itheo.large.sha}
	Given $\gamma\in\F_q^\times$, as $a\to\infty$, we have $|\sha(E_{\gamma, a})| = H(E_{\gamma, a})^{1+o(1)}$.
	\item 
	For any $\gamma\in\F_q^\times$ and $a\geq 1$, the $p$-primary part of $\sha(E_{\gamma, a})$ is trivial.
	\end{enumerate}
	\end{itheo}

\noindent
One can restate {\it (3)} as follows:   given $\epsilon>0$,  for any $\gamma\in\F_q^\times$ and any large enough integer $a\geq 1$ (depending on $\epsilon$), we have
\[ H(E_{\gamma, a})^{1-\epsilon} 
\leq |\sha(E_{\gamma, a})| 
\leq H(E_{\gamma, a})^{1+\epsilon}.\]
It is then clear that, for a given $\epsilon>0$, all but finitely many of the $E_{a, \gamma}$'s satisfy $|\sha(E_{\gamma, a})|\geq H(E_{\gamma, a})^{1-\epsilon}$.
Further, we will see that $H(E_{\gamma, a})=N(E_{\gamma, a})^{1/4}$, so that  $|\sha(E_{\gamma, a})|\geq N(E_{\gamma, a})^{1/4-\epsilon}$ holds for all but finitely many of the $E_{\gamma, a}$'s. 
In particular, Theorem \ref{itheo.large.sha} does imply Theorems \ref{itheo.sha.H} and \ref{itheo.sha.N}.
%The elliptic curves $E_{\gamma, a}$ are not ``semistable enough'', which is why we do not obtain a stronger result towards Conjecture \ref{iconj.deW}\eqref{iconj.N}. 

\subsection{Organisation of the paper}
We now explain the overall strategy of the proof of Theorem \ref{itheo.large.sha} as we describe the plan of the paper.
The first section briefly recalls the definitions and main properties of the objects that will be used throughout the paper. 
In the last subsection~\S\ref{sec.large.sha.Frob}, we show how the Frobenius isogenies allow one to construct  sequences of elliptic curves with large Tate--Shafarevich groups (thus proving Conjecture \ref{iconj.deW}\eqref{iconj.H} above).  
We found it worthwhile to include this construction since the argument in~\S\ref{sec.large.sha.Frob} illustrates on a simple example the general structure of the proof of Theorem \ref{itheo.large.sha}.

The elliptic curves $E_{\gamma, a}$ are introduced in section \ref{sec.family}, where we also calculate some of their more easily accessible invariants: 
their height and conductor, their Tamagawa number, as well as the torsion subgroup of $E_{\gamma, a}(K)$. 
Item {\it (1)} of Theorem \ref{itheo.large.sha} is proved in Proposition \ref{prop.noniso}.
The curves $E_{\gamma, a}$ were first studied by Pries and Ulmer in \cite{UlmerPries}: we recall in~\S\ref{sec.construction} how they were constructed there, and some of the results of \cite{UlmerPries}.

\medskip
A large part of the proof of Theorem \ref{itheo.large.sha} is analytic, in that we rely on a detailed study of the relevant $L$-functions. 
Our first main goal is therefore to obtain  an explicit expression for the $L$-function of $E_{\gamma,a}$.
To that end, we introduce some notation in section \ref{sec.prelim.Lfunc}.
We define a certain finite set $P_q(a)$ of places of $K$ and, to each place $v\in P_q(a)$ we attach in \S\ref{sec.gaun.kln} two character sums over the residue field of~$K$ at $v$: a Gauss sum~$\gaun(v)$ and a Kloosterman sum~$\Kloosn_\gamma(v)$. 
We then show (see Theorem \ref{theo.Lfunc}) that, for any integer $a\geq 1$  and any $\gamma\in\F_q^\times$, the $L$-function $L(E_{\gamma, a}, T)$ of~$E_{\gamma, a}$ admits the following expression:
\begin{equation}\label{ieq.Lfunc}
L(E_{\gamma, a}, T) = \prod_{v\in P_q(a)} \left(1-\gaun(v)\Kloosn_\gamma(v)\cdot T^{\deg v} +\gaun(v)^2 q^{\deg v}\cdot T^{2\deg v}\right).	
\end{equation}
This identity is proved in section \ref{sec.Lfunc} by an elementary method based on the definition of $L(E_{\gamma, a}, T)$ and manipulation of character sums.
This result, which is instrumental in the proof of our main theorem, may be of independent interest.

Using \eqref{ieq.Lfunc} and arithmetic properties of Gauss and Kloosterman sums, we elucidate in Theorem~\ref{theo.Lfunc.slopes} the $p$-adic valuations of the zeros of $L(E_{\gamma, a}, T)$.
As a consequence, we will deduce that $L(E_{\gamma, a}, T)$ does not vanish at the central point $T=q^{-1}$ for its functional equation (see Theorem~\ref{theo.nonvanishing.Lfunc}).
This non-vanishing result is enough to ensure that the BSD conjecture  holds for~$E_{\gamma, a}$ (Corollary~\ref{coro.BSD}) which, in turn, has several important corollaries for our study.
First, the Tate--Shafarevich group $\sha(E_{\gamma, a})$ is indeed (unconditionally) finite, as claimed in Theorem \ref{itheo.large.sha}{\it(2)}.
Secondly, we derive from the BSD formula and from our computations in section \ref{sec.family} that
\begin{equation}\label{ieq.link.sha.L}
|\sha(E_{\gamma, a})|	= q^{-1}H(E_{\gamma, a})\cdot L(E_{\gamma, a}, q^{-1}).
\end{equation}
Lastly -- even though this is less central to our point -- the  Mordell--Weil group $E_{\gamma, a}(K)$ is finite and, given our description of its torsion subgroup, we conclude that $E_{\gamma, a}(K) =\big\{\O, (0,0)\big\}$.

Given the link \eqref{ieq.link.sha.L} between the order of the Tate--Shafarevich group and the central value of $L(E_{\gamma, a}, T)$, 
we estimate the size of $|\sha(E_{\gamma, a})|$ in terms of $H(E_{\gamma, a})$ by proving adequate upper and lower bounds on~$L(E_{\gamma, a}, q^{-1})$.
Specifically, in order to show that Theorem \ref{itheo.large.sha}{\it(3)} holds we need to prove that 
\begin{equation}\label{ieq.goal.L}
	-o(1) \leq \frac{\log L(E_{\gamma, a}, q^{-1})}{\log H(E_{\gamma, a})} \leq o(1) \qquad (\text{as }a\to\infty).
\end{equation}
The proof of these inequalities is carried out in section \ref{sec.bnd.spval}.
Proving the lower bound in \eqref{ieq.goal.L} is the crucial step. 
After evaluating  expression  \eqref{ieq.Lfunc} for $L(E_{\gamma, a}, T)$ at $T=q^{-1}$, straightforward analytical considerations yield  that $L(E_{\gamma, a}, q^{-1})$  can be bounded from below by 
\[ \log L(E_{\gamma, a}, q^{-1})  
\geq \sum_{v\in P_q(a)} w\left(\frac{\Kloosn_\gamma(v)}{2{q}^{\deg v/2}}\right),
\quad\text{ where } w(x) := \log|x^2(1-x^2)|.\] 
For any place $v\in P_q(a)$, the algebraic number ${\Kloosn_\gamma(v)}/{2{q}^{\deg v/2}}$ is known to be totally real; 
moreover, its image in any complex embedding of $\Qbar$  lies in the interval~$[-1,1]$ by the Weil bound for Kloosterman sums. 
It is apparent that the size of the right-hand side of the previous display depends on how the set ${\Theta_{\gamma,a}:=\{ \Kloosn_\gamma(v)/{2{q}^{\deg v/2}}\}_{v\in P_q(a)}}$ is distributed in  $[-1,1]$.
We study the relevant aspects of the distribution of $\Theta_{\gamma, a}$ in section \ref{sec.distribution}.
Namely, we first recall from \cite{Griffon_LegAS} an effective version of an asymptotic distribution statement, 
the qualitative form of which says that, asymptotically as $a\to\infty$, the set $\Theta_{\gamma, a}$ becomes equidistributed in $[-1,1]$ with respect to the measure $\frac{2}{\pi}\sqrt{1-x^2}\dd x$.
Next we prove a quantitative version of the fact that $\Theta_{\gamma, a}$ ``avoids'' the points $-1, 0$ and $1$, which are the poles of $x\mapsto w(x)$ on $[-1,1]$. 
We then conclude the proof of the lower bound in \eqref{ieq.goal.L} by coupling these two facts about $\Theta_{\gamma, a}$ with a suitable approximation of $w$ on $[-1,1]$. 
The resulting estimates \eqref{ieq.goal.L} constitute Theorem \ref{theo.bounds.spval.L}. 
We combine these to \eqref{ieq.link.sha.L} in \S\ref{sec.large.sha} to conclude the proof of Theorem~\ref{itheo.large.sha}{\it(3)}.

Finally, in order to prove Theorem \ref{itheo.large.sha}{\it(4)}, we show that $|\sha(E_{\gamma, a})|$ is relatively prime to $p$ (see Theorem~\ref{theo.ppart.sha}).
To do so, we  start from the BSD formula \eqref{ieq.link.sha.L}: instead of estimating the ``archimedean size'' of the central value $L(E_{\gamma, a}, q^{-1})$ as we did above, we now estimate its ``$p$-adic size''.
The $p$-adic valuation of $L(E_{\gamma, a}, q^{-1})$ visibly depends on the $p$-adic valuations of the zeros of $L(E_{\gamma, a}, T)$. 
 The above mentioned Theorem \ref{theo.Lfunc.slopes} will allow us to obtain the desired result, in \S\ref{sec.ppart.sha}.

 %% -- %% -- %% -- %% -- %% -- %% -- %% -- %% -- %% -- %% -- %% -- %% -- %% -- %% -- %% -- %% -- %% -- %% -- %% -- %% -- %% -- %% -- %% -- %% -- %% -- %%     
\numberwithin{equation}{section}    

\section{Invariants of elliptic curves over function fields}
\label{sec.intro.invariants}

In this section, we introduce the relevant invariants of elliptic curves over function fields.

Let $\F_q$ be a finite field of characteristic $p$ and $K:=\F_q(t)$ denote the field of rational functions on the projective line $\P^1_{/\F_q}$. 
We assume throughout that $p\neq 2$. 
We do not mention  $K$ in the notation used for the invariants of elliptic curves over $K$ which we consider: unless explicitly noted otherwise, these invariants are relative to  $K$.

For a more detailed introduction to the arithmetic of elliptic curves over function fields and their invariants, the reader is referred to \cite{UlmerParkCity}.
Some of the results quoted in this section are valid more generally for elliptic curves (or even abelian varieties) over a function field of positive characteristic, but we only state the special cases that are required for our purpose.

\subsection{Conductor and height}
\label{sec.def.invariants}

Let $E$ be a non-isotrivial elliptic curve over $K$ (\ie{}, its $j$-invariant $j(E)$ is not a constant rational function).  
We denote by $\Delta_{\min}(E)$ its minimal discriminant divisor, which is a divisor on $\P^1_{/\F_q}$. 
The \emph{exponential differential height $H(E)$} is defined by
\[H(E) := q^{\frac{1}{12}\cdot\deg\Delta_{\min}(E)}.\]
The conductor divisor $\cond(E)$ of $E$ is also a divisor on $\P^1_{/\F_q}$ (see \cite[Chap.\ IV, \S10]{ATAEC} for the definition). 
We let 
\[N(E) := q^{\deg\cond(E)}.\]
We will call $N(E)$ the ``\emph{(numerical) conductor}'' of $E$.
It follows from an inequality between the conductor and the minimal discriminant divisors that $N(E)^{1/12}\leq H(E)$.
In the other direction,  Szpiro's inequality states that $H(E)\leq N(E)^{1/2}$, provided that $j(E)\in K$ is not a $p$-th power in $K$ (see \cite[Thm.\ 3]{GoldfeldSzpiro} and references in that article).
%Note that t
Two $K$-isogenous elliptic curves have the same numerical conductor (see \S\ref{sec.def.Lfunc}).

\subsection{The Tate--Shafarevich group}

Fix a separable closure $K^{\mathrm{sep}}$ of $K$ and consider the Galois cohomology group $\H^1(K,E):=\H^1(\Gal(K^{\mathrm{sep}}/K), E(K^{\mathrm{sep}}))$. %\[\H^1(K,E):=\H^1(\Gal(K^{\mathrm{sep}}/K), E(K^{\mathrm{sep}})).\] 
For any place $v$ of $K$, there is a similarly defined group $\H^1(K_v, E)$,
where $K_v$ denotes the completion of $K$ at $v$.
Recall that the \emph{Tate--Shafarevich group of~$E$} is defined by:
\begin{equation}\notag
    \sha(E) := \ker\left(\H^1(K,E) \longrightarrow \prod_{v}\H^1(K_v,E)\right),
\end{equation}
where the product runs over all places $v$ of $K$ and  the arrow is the product of the canonical restriction maps $\H^1(K, E)\to\H^1(K_v, E)$. 

Perhaps a more illuminating way of thinking about $\sha(E)$ is to observe that it is in bijection with the set of $K$-isomorphism classes of pairs $(C,\phi)$, where $C/K$ is a curve of genus $1$ which has at least one $K_v$-rational point for all places $v$ of $K$ and $\phi$ is a ${K}$-isomorphism between $E$ and $\mathrm{Jac}(C)$ (see \cite[Chap.\ X, \S3-4]{AEC} for more details). 
This point of view makes it clearer that $\sha(E)$ measures a local-global obstruction.

By construction, the Tate--Shafarevich group $\sha(E)$ is a torsion abelian group. 
It is conjectured, but not known in general, that $\sha(E)$ is finite (in the cases we study, we will show that it is indeed finite).

\subsection[L-functions]{The $L$-function} 
\label{sec.def.Lfunc}

For any place $v$ of $K$, we denote by $\F_v$ the residue field of $K$ at $v$ and by $\deg v=[\F_v:\F_q]$ the degree of $v$.
For such a $v$, let $(\widetilde{E})_v$ denote the reduction of $E$ at $v$. This plane cubic curve over $\F_v$ is not necessarily smooth: we say that $E$ has good reduction at $v$ if it is, and that $E$ has bad reduction at $v$ otherwise. 
In both cases, it makes sense to count $\F_v$-rational points on $(\widetilde{E})_v$.
For any place~$v$  put $a_v(E) := |\F_v| +1 - |(\widetilde{E})_v(\F_v)|$,
and define the \emph{$L$-function of $E$} by % the Euler product 
\begin{equation}\label{eq.defi.Lfunc.euler}
L(E, T) := \prod_{v}	L_v(E,T)^{-1},
\end{equation}
where the product runs over all places of $K$ and where
\begin{equation}\label{eq.defi.Lfunc}
L_v(E, T) = \begin{cases}
 1-a_v(E)\cdot T^{\deg v} +q^{d_v}T^{2\deg v} 
 & \text{if } E \text{ has good reduction at }v, \\
 1-a_v(E)\cdot T^{\deg v}
 & \text{if } E \text{ has bad reduction at }v.
 \end{cases}
\end{equation}
The Euler product in \eqref{eq.defi.Lfunc.euler} {\it a priori} converges on the complex disk $\big\{T\in\C : |T|<q^{-3/2}\big\}$, by the Hasse bound. 
However, deep results of Grothendieck and Deligne (among others) actually prove that $L(E,T)$ is a polynomial with integral coefficients in $T$, with constant coefficient $1$.
The degree $b(E)$ of that polynomial is given by the Grothendieck--Ogg--Shafarevich formula, which states that $\deg L(E, T) = \deg\cond(E) - 4$.
Moreover, the $L$-function satisfies the expected functional equation: $L(E,T) = \pm1 \cdot (qT)^{b(E)} \cdot L(E, 1/q^{2}T)$.
Lastly,  $L(E, T)$ satisfies the Riemann Hypothesis \ie{}, the complex zeros of $z\mapsto L(E, z)$ have magnitude~$q^{-1}$.

In particular, the value of $L(E,T)$ at the central point for its functional equation (\emph{viz.} $T=q^{-1}$) is a rational number, 
and the Riemann Hypothesis allows one to see that $L(E, q^{-1})$ is non-negative. 

Recall that two $K$-isogenous elliptic curves $E$ and $E'$ have the same $L$-function: see \cite[App.\ C]{Gross_bsd} for a proof.
In particular, their $L$-functions have the same degree and the Grothendieck--Ogg--Shafarevich formula implies that  $\deg\cond(E)=\deg\cond(E')$, 
so that the numerical conductors $N(E)$ and $N(E')$ coincide.

\medskip
By the results recalled in the previous paragraph, the $L$-function of $E$ can be written  as a product of the form
\[ L(E, T) = \prod_{k=1}^{b({E})}(1-z_k T ), \]
for some algebraic integers $z_k$ (which are the inverse zeros of the polynomial $L(E, T)$). 
We choose a prime ideal $\gP$ of $\Qbar$ above $p$ and  denote by  $\ord_\gP:\Qbar^\times\to\Q$ the corresponding discrete valuation, so normalised that $\ord_{\gP}(q)=1$. 
For any $k\in\{1, \dots, b(E)\}$, we let $\lambda_k =  \ord_\gP(z_k)\in\Q$. 
The functional equation of $L(E, T)$ implies that both $z_k$ and $q^2/z_k$ are algebraic integers, so that $0\leq \lambda_k\leq 2$ for all $k$. 
Renumbering the $z_k$'s if necessary, we can assume that $\lambda_1\leq \lambda_2 \leq \dots \leq \lambda_{b(E)}$. 
Thus ordered, the sequence $\{\lambda_1, \lambda_2, \dots, \lambda_{b(E)}\}$ will be called the \emph{$p$-adic slope sequence of $L(E, T)$}.
It then follows from the functional equation satisfied by $L(E, T)$ that one has  $\lambda_{b(E)-k} = 2- \lambda_k$ for any $k\in\{1, \dots, b(E)\}$.

\subsection{The BSD conjecture} 

Birch and Swinnerton-Dyer (and Tate, in this context) conjectured that the analytic behaviour of the $L$-function $z\mapsto L(E, z)$ around $z=q^{-1}$ encodes arithmetic information about~$E$. 
The reader is referred to \cite{Tate_BSD}, \cite[\S4]{GoldfeldSzpiro} or \cite{UlmerParkCity} for more detailed accounts of the BSD conjecture.

In the case that $L(E,T)$ does not vanish at $T=q^{-1}$, which is the most relevant for our purpose, their conjecture is entirely proved. 
We state the end result as follows: 

	\begin{theo}[BSD conjecture for elliptic curves of  analytic rank $0$]
	\label{conj.BSD}
	Let $E$ be non-isotrivial elliptic curve over~$K$. 
	Assume that $L(E, q^{-1})\neq 0$. 
	Then the following statements hold:
	\begin{enumerate}[(1)]
	\item 
		The Mordell--Weil group $E(K)$ is finite (\ie{}, $E(K)$ is torsion).		
	\item 
		The Tate--Shafarevich group $\sha(E)$ is finite. 		
	\item
		 Moreover, one has
		\begin{equation}\label{eq.conj.BSD}
			L(E,q^{-1}) = \frac{|\sha(E)|}{H(E)} \cdot \frac{q\cdot \tam(E)}{|E(K)|^2},
		\end{equation}
		where $\tam(E)$ denotes the Tamagawa number of $E$ (\ie{}, the product over all places $v$ of $K$ of the number of connected components of the fiber over $v$ of the N\'eron model of $E$).	
	\end{enumerate}	
	\end{theo}

We sketch a proof of this theorem.
Write $\rho=\ord_{T=q^{-1}}L(E, T)$ for the \emph{analytic rank of~$E$} \ie{}, the multiplicity of $q^{-1}$ as a root of $L(E, T)$.
Tate \cite{Tate_BSD} has proved that $0 \leq \rk_{\Z} E(K) \leq \rho$. 
By assumption  we have $\rho=0$, so that $\rk_\Z E(K)$ vanishes too: hence the finiteness of $E(K)$. 
Moreover, we have $\rho=\rk_{\Z}E(K)$ (this equality is the so-called the ``weak BSD conjecture''). 
It is further known, by deep results of Tate and Milne (see \cite{Tate_BSD, Milne_conjAT}), that ``weak BSD'' implies the whole BSD conjecture. 
Therefore, the equality $\rho=\rk_{\Z}E(K)$  implies {\it(2)} and {\it(3)}.

\subsection[Large sha via Frobenius isogenies]{Large $\sha$ via Frobenius isogenies}
\label{sec.large.sha.Frob}

This section is inspired by \cite[\S5]{Ulmer_BSalg}:
our main purpose here is to illustrate, on a very simple example, the steps that we will later follow to prove that some elliptic curves possess a large Tate--Shafarevich group. 

Let $E$ be an elliptic curve over $K=\F_q(t)$ and,  for any integer $n\geq 1$, let $E_n$ denote the base change of $E/K$ under the $p^n$-th power Frobenius morphism $\mathrm{Fr}_{p^n}:K\to K$.
In other words, we put $E_n = E\times_{K}K$ where the underlying map $K\to K$ is $\mathrm{Fr}_{p^n}$.
The induced map $F_{p^n}:E\to E_n$ is a purely inseparable isogeny of degree $p^n$ (see \cite[p.\ 225]{UlmerParkCity}).
We obtain in this manner a sequence $(E_n)_{n\geq 1}$ of elliptic curves over $K$. 
One obviously has $j(E_n)=j(E)^{p^n}$. 
Moreover, we have $N(E_n)=N(E)$ since the curves are $K$-isogenous (see \S\ref{sec.def.invariants}). 
For the same reason, we have $L(E_n, T)=L(E,T)$ for all $n\geq 1$.
On the other hand, the proof of Theorem 5.1 in \cite{Ulmer_BSalg} shows that the height $H(E_n)$ tends to infinity as $n\to\infty$.

Now let us \emph{assume that $E$ is chosen so that $L(E, q^{-1})$ is non-zero}.
Clearly, the central value $L(E_n, q^{-1})$ is then also non-zero for any $n\geq 1$.
This implies that the full BSD conjecture holds for all the elliptic curves~$E_n$ in the sequence (see Theorem \ref{conj.BSD}).
In particular, the Tate--Shafarevich group~$\sha(E_n)$ is finite, for all $n\geq 1$.
Furthermore, the BSD formula \eqref{eq.conj.BSD} states that 
\[L(E_n, q^{-1}) 
= \frac{|\sha(E_n)|}{H(E_n)}\cdot \frac{\tam(E_n)\cdot q}{|E_n(K)\tors|^2}.\]
The growth of $\tau(E_n)$ can be estimated in terms of $H(E_n)$: \cite[Thm.\ 1.22]{HP15} or \cite[Thm.\ 1.5.4]{Griffon_PHD} yield that $\log\tam(E_n) = o\big(\log H(E_n)\big)$ as $n\to\infty$.
Further, there is a uniform bound on torsion subgroups of elliptic curves over $K$ (see Proposition 7.1 in \cite[Lect.\ 1]{UlmerParkCity} for instance): here, this shows that $|E_n(K)\tors|$ remains bounded as $n\to\infty$. 
From the last displayed identity and these two bound, we deduce that:
\[\frac{\log|\sha(E_n)|}{\log H(E_n)} = 1+ \frac{\log L(E_n, q^{-1})}{\log H(E_n)}  + o(1)\qquad(\text{as }n\to\infty).\]
Since $L(E_n, q^{-1})=L(E, q^{-1})=O(1)$, we therefore obtain that ${\log|\sha(E_n)|}\sim{\log H(E_n)}$ as $n\to\infty$. 
In other words, we have $|\sha(E_n)| = H(E_n)^{1+o(1)}$ as $n\to\infty$, so that the elliptic curves $E_n$ do have ``large Tate--shafarevich groups''.

Let us denote by $\val_p:\Q^\times\to\Q$ the $p$-adic valuation on $\Q$, normalised so that $\val_p(q)=1$.
We now evaluate the $p$-adic valuations of both sides of the BSD formula (instead of their logarithm). 
The above mentioned upper bound on $\tau(E_n)$ also yields that $\val_p(\tau(E_n)) =O\big(\val_p(H(E_n))\big)$ as $n\to\infty$.
Hence, we have
\[\frac{\val_p(|\sha(E_n)|)}{\val_p (H(E_n))} \geq 1+ \frac{\val_p( L(E_n, q^{-1}))}{\val_p (H(E_n))} + o(1) = 1 + o(1).\]
Besides, the order of $\sha(E_n)[p^\infty]$ equals $q^{\val_p(|\sha(E_n)|)}$. 
Since $\val_p (H(E_n))$ is none other than $\log H(E_n)/\log q$, we deduce from the above that $|\sha(E_n)[p^{\infty}]|\geq H(E_n)^{1+o(1)}$ as $n\to\infty$.
Therefore, the ``large $\sha$'' phenomenon for the elliptic curves in the sequence $(E_n)_{n\geq 1}$ is ``explained'' by a large $p$-primary part of $\sha$.

\medskip
The set of elliptic curves $E$ to which this construction applies is non-empty.
For instance, consider the Legendre elliptic curve $E/K$ given by $y^2=x(x-1)(x-t)$. 
This curve is non-isotrivial, and a straighforward application of Tate's algorithm shows that the conductor divisor $\cond(E)$ has degree $4$. 
The Grothendieck--Ogg--Shafarevich formula then predicts that $\deg L(E, T)=0$.
This forces the $L$-function to be trivial (\ie{}, $L(E, T)=1$); in particular, $L(E, q^{-1})$ is non-zero.
The previous paragraph shows that we have:
    \begin{prop}\label{prop.large.sha.Frob} 
    For all $n\geq 1$, consider the elliptic curve $E_n$ defined over $K$ by 
    \[E_n:\qquad y^2=x\big(x-1\big)\big(x-t^{p^n}\big).\]
    For all $n\geq 1$, the Tate--Shafarevich group $\sha(E_n)$ is finite and, as $n\to\infty$, one has   
    \[|\sha(E_n)| = H(E_n)^{1+o(1)}.\]
    \end{prop}

\noindent
This proposition proves that Conjecture \ref{iconj.deW}\eqref{iconj.H} in the introduction is true. 
To prove Theorem \ref{itheo.large.sha}, we rely on a somewhat similar strategy: the various steps we follow in the rest of the paper are but more elaborate versions of the ones in the proof of Proposition \ref{prop.large.sha.Frob}.

    \begin{rema} 
     In case the characteristic $p$ of $K$ is congruent to $1$ modulo $6$, the authors of~\cite{GriffonUlmer} have very recently produced a sequence of elliptic curves over $K$ with large Tate--Shafarevich groups, the $p$-primary parts of which are trivial. 
     In contrast to the elliptic curves we study here, though, the elliptic curves forming the sequence of \cite{GriffonUlmer}, being  sextic twists of $y^2=x^3+1$ over $K$, are all isotrivial and are all $\bar{K}$-isomorphic (they have $j$-invariant $0$). 
     Theorem \ref{itheo.large.sha} is also independent of the congruence of $p$ modulo~$6$. 
      \end{rema}

\section{The elliptic curves $E_{\gamma, a}$} 
\label{sec.family}

Let $\F_q$ be a finite field of characteristic $p\geq 3$ and $K=\F_q(t)$. 
In what follows, we identify, as we may, finite places of $K$ with monic irreducible polynomials in $\F_q[t]$ and/or with closed points of $\mathbb{A}^1_{/\F_q}$. 

For any integer $a\geq 1$, we let $\wp_a(t):=t^{q^a}-t\in\F_q[t]$.

\subsection{Definition}
Let $\gamma$ be a nonzero element of $\F_q$ and $a\geq 1$ be an integer. 
We let $E_{\gamma, a}$ be the elliptic curve over $K$ given by the affine Weierstrass model
\begin{equation}\label{eq.Wmod}
E_{\gamma, a}:\qquad y^2=x^3+\wp_a(t)\cdot x^2 +\gamma x.
\end{equation}
Its $j$-invariant is easily computed to  be 
\begin{equation}\label{eq.jinv}
j(E_{\gamma, a}) = 2^8\cdot\frac{(\wp_a(t)^2-3\gamma)^3}{\gamma^2(\wp_a(t)^2-4\gamma)},
\end{equation}
and the discriminant of the model \eqref{eq.Wmod} is 
$\Delta_{\gamma, a} = 2^4\gamma^2\cdot (\wp_a(t)^2-4\gamma)$.
Note that the $j$-invariant is not constant, hence the curve $E_{\gamma, a}$ is non-isotrivial.

\subsection{Construction of $E_{\gamma, a}$}
\label{sec.construction}

The sequences $\{E_{\gamma, a}\}_{a\geq 1}$ are essentially some of the titular Artin--Schreier families of elliptic curves studied in \cite{UlmerPries} (see \S6.2 there).
We briefly give more details about how the sequences $\{E_{\gamma, a}\}_{a\geq 1}$ were constructed in {\it loc.\ cit.} and ``compare'' them to some  other similar such sequences. 

%One of the main points of \cite{UlmerPries} is to exhibit a number of families of elliptic curves over $K$ for which the BSD conjecture holds.
The basic input of the construction in \cite{UlmerPries} is a pair $(f,g)$ of rational functions $\P^1_{/\F_q}\to\P^1_{/\F_q}$ whose divisors satisfy mild conditions. 
The output is a sequence, indexed by powers $Q$ of~$q$, 
of smooth projective surfaces $X_Q$ over $\F_q$,  each equipped with a natural fibration $\pi_Q:X_Q\to\P^1$ and admitting a dominant rational map $C_{f, Q}\times C_{g,Q} \dashrightarrow X_Q$ from the product of two curves $C_{f, Q}$, $C_{g, Q}$  defined over~$\F_q$.
One of the main results of \cite{UlmerPries} is that the Jacobian variety of the generic fiber of $\pi_Q$ satisfies the BSD conjecture. %this construction produces a number of families of elliptic curves over $K$ for which the BSD conjecture holds.

Pries and Ulmer further classify (see Proposition 3.1.5 and \S4.2 in \cite{UlmerPries}) the various polar behaviours of $f$ and $g$ for which the resulting sequence $(X_Q)_{Q}$ is a sequence of elliptic surfaces (\ie{}, for which the generic fiber of $\pi_Q$ is an elliptic curve over $K$). 
This classification results in seven ``types'' of pairs $(f, g)$.

\medskip
The elliptic curve $E_{\gamma, a}$ defined by \eqref{eq.Wmod} corresponds to the case $(2, 1+1)$ of their construction. 
Specifically, one starts with $f(u)=u^2$ (one double pole at $\infty\in\P^1$) and $g(v)=v+\gamma/v$ (one simple pole at $0$ and one simple pole at $\infty$).
For any power $Q$ of $q$, one  obtains a smooth projective surface $X_Q$ over $\F_q$ which is birational to the affine surface $Y_Q$ defined in affine coordinates $(u,v,t)$ by 
\[Y_Q:\qquad  f(u) - g(v) = t^Q -t.\]
The surface $X_Q$ is equipped with a surjective morphism  $\pi_Q:X_Q\to\P^1$, which extends the natural projection $Y_Q\to\mathbb{A}^1$ given by $(u,v,t)\mapsto t$.
Writing that $Q=q^a$ for some integer $a\geq 1$, the elliptic curve $E_{\gamma, a}$ over $K$ then arises as the generic fiber of $\pi_Q$.
The model \eqref{eq.Wmod} for  $E_{\gamma, a}/K$ indeed appears naturally by putting 
the equation for the generic fiber of $\pi_Q$ into Weierstrass form.

\medskip
We wish to remark that one other case of the Pries--Ulmer construction is, for a good choice of parameters, ``isogenous'' to the case $(2, 1+1)$. 
For any $\gamma\in\F_q^\times$, consider the two rational maps $f,g:\P^1\to\P^1$ given by
$f(u) = u^2$ and $g(v) = v^2+ \gamma /v^2$.
In the terminology of \cite{UlmerPries}, this input data has type $(2, 2+2)$ since $f$ has a single double pole at $\infty$, and $g$ has double poles at $0$ and at $\infty$.
For any power $Q$ of $q$, let $X_Q$ be the smooth projective surface over $\F_q$, fibered over $\P^1$ {\it via} $\pi_Q:X_Q\to\P^1$, which is birational to the affine surface $Y_Q$ over $\F_q$ given by 
\[Y_Q : \qquad f(u)-g(v) = t^Q-t\]
and such that the morphism $\pi_Q:X_Q\to\P^1$ extending $(u,v,t)\in Y_Q\mapsto t\in\mathbb{A}^1$ is {minimal}. 
Writing  $Q=q^a$ for some integer  $a\geq 1$, we denote by $E'_{\gamma, a}/K$ the generic fiber of $\pi_Q$. 
By construction, after clearing denominators, the curve $E'_{\gamma, a}$ is given in affine coordinates $(u,v)\in\mathbb{A}^2_{/K}$ by %$u^2v^2 = v^4 +\wp_a(t)\cdot v^2 +\gamma$. 
\[E'_{\gamma, a} : \qquad u^2v^2 = v^4 +\wp_a(t)\cdot v^2 +\gamma.\]
The change of coordinates 
$(u,v)\mapsto (x,y)=\left(2v(u+v), 4v\big(uv+v^2+\wp_a(t)\big) \right)$
then provides the following affine Weierstrass model for $E'_{\gamma,a}$ over $K$:
\[E'_{\gamma, a}:\qquad y^2=\big(x+\wp_a(t)\big)\big(x^2-4\gamma\big).\]
Our main observation is:

	\begin{prop}
	For any $\gamma\in\F_q^\times$ and any integer $a\geq 1$,  the elliptic curve $E'_{\gamma, a}$ defined over $K$ by the affine Weierstrass model
	\[E'_{\gamma, a}:\qquad y^2=\big(x+\wp_a(t)\big)\big(x^2-4\gamma\big)\]
	is $2$-isogenous to $E_{\gamma, a}$ over $K$.
	\end{prop}
	\begin{proof} 
	Using the formulae in \cite[Chap.\ III, Ex.\ 4.5]{AEC}, we find that the map
	\[\phi : (x,y)\in E_{\gamma, a}\mapsto 
	\left({y^2/x^{2}-\wp_a(t)}, {y\big(1-\gamma/x^{2}\big)}\right)\in E'_{\gamma, a},\]
	which is clearly defined over $K$, provides the desired $2$-isogeny.
	\end{proof}

Since they are $K$-isogenous, the elliptic curves $E_{\gamma, a}$ and $E'_{\gamma, a}$ have the same conductor, 
they share the same $L$-function, 
and their Mordell-Weil ranks are equal. 
Therefore, our main results about the elliptic curves $E_{\gamma, a}$ (such as the explicit expression for their $L$-function, and the fact that they have large Tate-Shafarevich groups) are also valid for the curves $E'_{\gamma, a}$.

Furthermore, we know by \cite[Coro.\ 3.1.4]{UlmerPries}  that, for all $\gamma\in\F_q^\times$ and $a\geq 1$, the elliptic curves $E_{\gamma, a}$ and~$E'_{\gamma, a}$ satisfy the BSD conjecture.
The calculations of \cite[\S6.2]{UlmerPries} further show that the Mordell-Weil group $E_{\gamma, a}(K)$ has rank $0$.
We will recover these two facts \emph{via} a different method, as corollaries of our explicit expression for the $L$-function of $E_{\gamma, a}$ (see Theorem \ref{theo.nonvanishing.Lfunc} and Corollary \ref{coro.BSD}).

\subsection{Bad reduction and invariants}
\label{sec.invariants}

We let $B_{\gamma, a}$ denote the set of finite places of $K$ corresponding to monic irreducible polynomials in $\F_q[t]$ which divide $\wp_a(t)^2-4\gamma$. 
Equivalently, $B_{\gamma, a}$ is the set of finite places of $K$ dividing the discriminant of the Weierstrass model \eqref{eq.Wmod} for $E_{\gamma, a}$. 

	\begin{prop}\label{prop.bad.red}
%	In the above setting,
	For any $\gamma\in\F_q^\times$ and any integer $a\geq1$, the elliptic curve $E_{\gamma, a}$ has good reduction outside $B_{\gamma, a}\cup\{\infty\}$. 
	Moreover, its places of bad reduction are as follows:
	\vspace{-6pt}\begin{center}
	\renewcommand{\arraystretch}{2.0}
	\begin{tabular}{ c  l   c  c } 
	Place $v$&   Reduction at $v$  & $\delta_{v}(E_{\gamma, a})$ & $\nu_{v}(E_{\gamma, a})$  \\
	\hline \hline
	$v\in B_{\gamma,a}$ & Multiplicative (fiber of type $\mathbf{I}_{1}) $   & $1$ & $1$  \\ \hline
	$\infty$    & Additive (fiber of type $\mathbf{I}^\ast_{4q^a}$)    &    $4q^a + 6$ & $2$  \\ \hline
	\end{tabular}
	\end{center}
	\renewcommand{\arraystretch}{1.0}  
	where, for any place $v$ of $K$, we have denoted by $\delta_v(E_{\gamma, a})$ (resp. by $\nu_v(E_{\gamma, a})$) the order at $v$ of the minimal discriminant divisor of $E_{\gamma,a}$ (resp. of its conductor divisor).
	\end{prop}
	\begin{proof} 
		This follows  from a routine application of Tate's algorithm to compute the type of the fibers of bad reduction  (as explained in  \cite[\S IV.9]{ATAEC}).
	\end{proof}

From the above proposition, we deduce explicit expressions for the exponential differential height $H(E_{\gamma, a})$ and the ``numerical'' conductor $N(E_{\gamma, a})$ of $E_{\gamma, a}$ (as defined in \S\ref{sec.def.invariants}):
	\begin{equation}\label{eq.invariants}
	H(E_{\gamma, a}) 
	= q^{\frac{1}{12} \deg\Delta_{\min}(E_{\gamma, a})} = q^{(q^a+1)/{2}}, \qquad \text{ and } \qquad
	N(E_{\gamma, a}) 
	=q^{\deg\cond(E_{\gamma, a})}= q^{2(q^a+1)}.  
	\end{equation}
Indeed, the polynomial $\wp_a(t)^2-4\gamma\in\F_q[t]$ being squarefree,  we see that 
\[\sum_{v\in B_{\gamma, a}} \deg v 
=\sum_{v\mid \wp_a(t)^2-4\gamma} \deg v 
= \deg\big(\wp_a(t)^2-4\gamma\big) 
= 2 q^a.\]
It is clear from \eqref{eq.invariants} that we have $H(E_{\gamma, a}) = N(E_{\gamma, a})^{1/4}$.

The \emph{Tamagawa number} $\tau(E_{\gamma,a})$ of $E_{\gamma, a}$ is the product over all places $v$ of the number of components in the special fiber of the N\'eron model of $E_{\gamma, a}$ at $v$. 
It is also readily calculated from the above proposition: with the help of the table on p.\ 365 of \cite{ATAEC}, one obtains that
	\begin{equation}\label{eq.tamagawa}
	\tam(E_{\gamma, a})=4.
	\end{equation}

	\begin{rema}\label{rema.mini.mod} 
	For the computation of the $L$-function in section \ref{sec.Lfunc}, it will be useful to have minimal integral models of $E_{\gamma, a}$ at finite places of $K$ at our disposal.
	For a finite place $v$ of $K$, by comparing the valuation at $v$ of the discriminant $\Delta_{\gamma, a}$ of the model \eqref{eq.Wmod} with the valuation at $v$ of the minimal discriminant in the above table, 
	we find that the Weierstrass model \eqref{eq.Wmod} is minimal and integral at $v$.
	\end{rema}

As an immediate corollary of this computation of the invariants of $E_{\gamma, a}$, we obtain:

	\begin{prop}\label{prop.noniso} 
	As $\gamma$ varies in $\F_q^\times$ and as $a\geq1$ varies among integers, 
	the elliptic curves $E_{\gamma, a}$ are pairwise non $\bar{K}$-isomorphic and pairwise non $K$-isogenous.	
	\end{prop}
	\begin{proof} Two elliptic curves over $K$ are $\bar{K}$-isomorphic if and only if they have the same $j$-invariant. 
	It is obvious from \eqref{eq.jinv} that $\deg j(E_{\gamma, a})=4a$	is strictly increasing when $a\geq 1$ grows.
	Also apparent on \eqref{eq.jinv} is the fact that, for a fixed $a$, the position of the poles of $j(E_{\gamma, a})$ 	varies with $\gamma\in\F_q^\times$.
	Hence the first assertion.
	
	In a similar vein, it follows from \eqref{eq.invariants} that $\deg\cond(E_{\gamma, a})$ increases with $a\geq 1$.  
	Since two $K$-isogenous elliptic curves have equal conductor divisors, we conclude that $E_{\gamma, a}$ and $E_{\gamma, a'}$  cannot be $K$-isogenous for any $a\neq a'$.
	For a given $a\geq 1$ and distinct $\gamma, \gamma'\in\F_q^\times$, the curves $E_{\gamma, a}$ and $E_{\gamma', a}$ are not $K$-isogenous either, 
	since the sets $B_{\gamma, a}\cup\{\infty\}$ and $B_{\gamma', a}\cup\{\infty\}$ of their places of bad reduction differ (\ie{} their conductor divisors have different supports).
	\end{proof}

\subsection{Torsion subgroup}
\label{sec.torsion}

We conclude this section by elucidating the structure of the torsion subgroup of the Mordell-Weill group $E_{\gamma,a}(K)$:

    \begin{theo}
    \label{theo.tors}
    For any $\gamma\in\F_q^\times$ and any integer $a\geq 1$, 
    let $P_0\in E_{\gamma, a}(K)$ be the point with coordinates $(0,0)$ in the Weierstrass model \eqref{eq.Wmod}. 
    Then we have $E_{\gamma, a}(K)\tors =\big\{\O, P_0\big\}$. %    \begin{equation*}    E_{\gamma, a}(K)\tors =\big\{\O, P_0\big\}     \simeq \Z/2\Z.    \end{equation*}
    \end{theo}

\noindent
We will show later on (see Corollary \ref{coro.BSD}) that the whole group $E_{\gamma,a}(K)$ is torsion, so that the above result actually provides the complete list of $K$-rational points on $E_{\gamma,a}$.

    \begin{proof}
	We see on \eqref{eq.jinv} that the $j$-invariant of $E_{\gamma, a}$ is not a $p$-th power in $K$, hence 
    Proposition 7.1 in \cite[Lect.\ I]{UlmerParkCity} ensures that $E_{\gamma, a}(K)\tors$ contains no point with $p$-th power order.

   We have shown in the previous subsection that  $E_{\gamma, a}$ has additive reduction of type $\mathbf{I}^\ast_{4q^a}$ at $\infty$. 
	By \cite[Lem.\ 7.8]{SchShio},  the prime-to-$p$ part of $E_{\gamma, a}(K)\tors$ embeds into the group $G_\infty$ of components of the special fiber at $\infty$ of the N\'eron model of $E_{\gamma, a}$. 
    Since the reduction at $\infty$ is of type $\mathbf{I}^\ast_{4q^a}$, the table in \S7.2 of \cite{SchShio} tells us that $G_\infty\simeq(\Z/2\Z)^2$.
    Hence, $E_{\gamma, a}(K)\tors$ is isomorphic to a subgroup of $(\Z/2\Z)^2$.
    In particular, the torsion subgroup of $E_{\gamma, a}(K)$ consists only of $2$-torsion points and we infer that $E_{\gamma, a}(K)\tors = E_{\gamma, a}(K)[2]$.

    The ($\bar{K}$-rational) $2$-torsion subgroup of $E_{\gamma, a}$ is readily computed: it consists of $4$ points given, in the homogenised version of \eqref{eq.Wmod}, by
    \begin{multline*} 
    	\O=[0:1:0], \quad
    	P_0 = [0:0:1], \\
        P_+ = \left[  \wp_a(t) + \sqrt{\wp_a(t) - 4\gamma}:0:-2\right],  \quad
        P_- = \left[  \wp_a(t) - \sqrt{\wp_a(t) - 4\gamma}:0:-2\right].
    \end{multline*}
    Since $\wp_a(t)-4\gamma\in\F_q[t]$ is squarefree, only the first two points are $K$-rational
    (the latter two are rational over the quadratic extension $K(\sqrt{\wp_a(t)-4\gamma})$ of $K$). 
    We deduce that $E_{\gamma, a}(K)[2] = \{\O, P_0\}$.

    By the previous paragraph, the proof is complete.
    \end{proof}

%% -- %% -- %% -- %% -- %% -- %% -- %% -- %% -- %% -- %% -- %% -- %% -- %% -- %% -- %% -- %% -- %% -- %% -- %% -- %% -- %% -- %% -- %% -- %% -- %%
\section{Preliminaries on character sums}
\label{sec.prelim.Lfunc}

In the next section (see Theorem \ref{theo.Lfunc}), we will compute an explicit expression for the $L$-function of the curves $E_{\gamma, a}$ introduced above.  
To carry out this computation, we first need to set up some notation and conventions. 
These will be in force for the rest of the paper.

\subsection{Gauss sums and Kloosterman sums}
\label{sec.gauss.kloos}

Let $\F$ be a finite field of odd characteristic $p$.
Any additive character $\psi$ on $\F$ may and will be assumed to take values in the $p$-th cyclotomic field $\Q(\zeta_p)$. 
For any finite extension $\F'/\F$, we denote the trace map by $\trace_{\F'/\F}:\F'\to\F$.
If $\psi$ is an additive character on $\F$, then the composition $\psi\circ\trace_{\F'/\F}$ is an additive character on $\F'$.
 
We denote by $\lambda:\F^\times\to\Qbar^\times$ (or $\lambda_\F$ if confusion is likely to arise) the unique nontrivial multiplicative character of order $2$ on $\F^\times$. 
We extend $\lambda$ to the whole of $\F$ by setting $\lambda(0):=0$.

    \begin{defi} 
    For any additive character $\psi$ on $\F$, define the \emph{quadratic Gauss sum} $\gauss{\F}{\psi}{\lambda}$ by 
    \begin{equation*} 
        \gauss{\F}{\psi}{\lambda} := -\sum_{x\in\F}\lambda(x)\cdot\psi(x).
    \end{equation*}
    \end{defi}

\noindent 
Note our choice of normalising $\gauss{\F}{\psi}{\lambda}$ by multiplying the sum by $-1$. 
By construction, the sum 
$\gauss{\F}{\psi}{\lambda}$ is an algebraic integer in the cyclotomic field $\Q(\zeta_p)$.
Recall the following facts about Gauss sums:
    \begin{enumerate}[(\sf {Ga}\,1)]
    \item\label{item.gauss.orbit}
        For any nontrivial additive character $\psi$ on $\F$, any $\alpha\in\F^\times$ and any finite extension $\F'/\F$, define $\psi_{\F'}^{(\alpha)}$ by $x\mapsto \psi \circ\trace_{\F'/\F}(\alpha\cdot x)$. 
        The map $\psi_{\F'}^{(\alpha)}$ is a nontrivial additive character on $\F'$ and, letting $\alpha' := \alpha^{|\F|}$, one has
        \[\gauss{\F'}{\psi_{\F'}^{(\alpha')}}{\lambda_{\F'}} = \gauss{\F'}{\psi_{\F'}^{(\alpha)}}{\lambda_{\F'}}.\]

    \item\label{item.gauss.HD} %Hasse--Davenport for Gauss sums
        For any nontrivial additive character $\psi$ on $\F$ and any finite extension $\F'/\F$, one has 
        \[\gauss{\F'}{\psi\circ\trace_{\F'/\F}}{\lambda_{\F'}} = \gauss{\F}{\psi}{\lambda_\F}^{[\F':\F]}.\]

    \item\label{item.gauss.magnitude} %Magnitude of Gauss sums
        For any nontrivial additive character $\psi$ on $\F$, one has $|\gauss{\F}{\psi}{\lambda}| = |\F|^{1/2}$ 
        in any complex embedding of $\Q(\zeta_p)$.

    \item\label{item.gauss.quad} %Explicit formula for quadratic Gauss sums
        For any nontrivial additive character $\psi$ on $\F$, the quotient $\gauss{\F}{\psi}{\lambda}/|\F|^{1/2}$ is a $4$th root of unity (which can be explicitly determined).
    \end{enumerate}
These results are quite classical, and the reader is referred to \cite[Chap.\ V, \S2]{LidlN} for proofs.

    \begin{defi} For  $\alpha\in\F$ and an additive character $\psi$ on $\F$,  define the \emph{Kloosterman sum} $\kloos{\F}{\psi}{\alpha}$ by
    \begin{equation}
        \kloos{\F}{\psi}{\alpha} := -\sum_{x\in\F^\times}\psi\left(x + \frac{\alpha}{x}\right).
    \end{equation}
    \end{defi}

\noindent
Again, we point out our choice of normalising the sum by multiplying it by $-1$.
One can show that the Kloosterman sum $\kloos{\F}{\psi}{\alpha}$ is a totally real algebraic integer in $\Q(\zeta_p)$ \ie{}, 
$\kloos{\F}{\psi}{\alpha}\in\Z[\zeta_p+\zeta_p^{-1}]$.
We remind the reader of the following facts about Kloosterman sums: 
    \begin{enumerate}[(\sf {Kl}\,1)]
    \item\label{item.kloos.salie} %Sali\'e's formula
        For any nontrivial additive character $\psi$ on $\F$ and any $\alpha\in\F^\times$, one has the identity:
        \[\kloos{\F}{\psi}{\alpha} =  - \sum_{y\in\F} \lambda(y^2-4\alpha)\cdot\psi(y).\]

    \item\label{item.kloos.orbit}
        For any nontrivial additive character $\psi$ on $\F$, any $\alpha\in\F$ and any finite extension $\F'/\F$, letting $\alpha'=\alpha^{|\F|}$, one has  
        \[\kloos{\F'}{\psi\circ\trace_{\F'/\F}}{\alpha'} = \kloos{\F'}{\psi\circ\trace_{\F'/\F}}{\alpha}.\]

    \item\label{item.kloos.HD} %Hasse--Davenport for Kloosterman sums
        For any nontrivial additive character $\psi$ on $\F$ and any $\alpha\in\F^\times$, there is a unique pair $\{\klo_\F(\psi;\alpha), \klo'_\F(\psi;\alpha)\}$ of  algebraic integers, whose product is $|\F|$ and such that,  for any finite extension $\F'/\F$, one has
        \[\kloos{\F'}{\psi\circ\trace_{\F'/\F}}{\alpha} = \klo_{\F}(\psi;\alpha)^{[\F':\F]} + \klo'_{\F}(\psi;\alpha)^{[\F':\F]}. \]

    \item\label{item.kloos.magnitude} 
        For any nontrivial additive character $\psi$ on $\F$ and any $\alpha\in\F^\times$, one has 
        $|\klo_{\F}(\psi;\alpha)| = |\klo'_{\F}(\psi;\alpha)| = |\F|^{1/2}$ in any complex embedding of $\Qbar$. 
        In particular, one has $|\kloos{\F}{\psi}{\alpha}| \leq 2\cdot |\F|^{1/2}$ % \[|\kloos{\F}{\psi}{\alpha}| \leq 2\cdot |\F|^{1/2} \]
        in any complex embedding of $\Q(\zeta_p)$.

    \end{enumerate}
These results are classical: see \cite[Chap.\ V, \S5]{LidlN} for  proofs thereof. 
For convenience, we also state here a fact that will only be proved later on 
(see Lemma \ref{lemm.kloos.slope} and Remark \ref{rema.nonvanishing.reloaded}):
    \begin{enumerate}[(\sf {Kl}\,1)]
    \setcounter{enumi}{4}
    \item\label{item.kloos.avoid}
        For any nontrivial additive character $\psi$ on $\F$ and any $\alpha\in\F^\times$, $\kloos{\F}{\psi}{\alpha}$ is a $p$-adic unit in~$\Q(\zeta_p)$. 
        
        In particular, one has $0 < |\kloos{\F}{\psi}{\alpha}|< 2\cdot |\F|^{1/2}$ % \[ 0 < |\kloos{\F}{\psi}{\alpha}|< 2\cdot |\F|^{1/2} \]
        in any complex embedding of $\Q(\zeta_p)$.
    \end{enumerate}

\subsection{Places of degree dividing $a$}
\label{sec.places.a}

Let $\F_q$ be a finite field of odd characteristic, and let $K:=\F_q(t)$ denote the rational function field over $\F_q$.

    \begin{defi}\label{defi.Pqa}
    For any integer $a\geq 1$, we denote by $P_q(a)$ the set of places $v$ of $K$ such that $v\neq0, \infty$ and $\deg v\mid a$.
    That is to say, $P_q(a)$ is the set of closed points of the multiplicative group $\mathbb{G}_{m}=\P^1\smallsetminus\{0, \infty\}$ over $\F_q$ whose degree divides $a$.
    In the usual identification between finite places of $K$ and monic irreducible polynomials in $\F_q[t]$, the set $P_q(a)$ corresponds to the set of monic irreducible polynomials $B\in\F_q[t]$ such that $B\neq t$ and $\deg B\mid a$.
    \end{defi}

The latter interpretation allows for an easy estimation of the cardinality $|P_q(a)|$. 
Indeed the Prime Number Theorem for~$\F_q[t]$ states that, for any integer $n\geq 1$, the number $\pi_q(n)$ of monic irreducible polynomials in $\F_q[t]$ of degree $n$ satisfies: 
$q^n/n - q^{n/2}\leq \pi_q(n) \leq q^n/n$ 
(see, for instance, \cite[Thm.\ 2.2]{Rosen} and its proof). 
Noting that $|P_q(a)| = -1 +\sum_{n\mid a} \pi_q(n)$,  we deduce the existence of constants $c_q, c'_q>0$, depending at most on $q$, such that
    \begin{equation}\label{eq.estimate.Pqa}
    \forall a\geq 1, \qquad     
    c'_q \cdot \frac{q^a}{a} \leq |P_q(a)|  \leq c_q \cdot \frac{q^a}{a}.
    \end{equation}

\subsection[The sums]{The sums $\gaun(v)$ and $\Kloosn_\gamma(v)$}
\label{sec.gaun.kln}

Fix a finite field $\F_q$ of odd characteristic $p$, and endow $\F_q$ with a nontrivial additive character $\psi_q$ taking values in the $p$-th cyclotomic field $\Q(\zeta_p)$.
For any finite extension~$\F/\F_q$ we ``lift'' $\psi_q$ to a nontrivial character $\psi_\F$ on $\F$ by composing $\psi_q$ with the relative trace; 
\ie{}, we let $\psi_{\F}:=\psi_q \circ \trace_{\F/\F_q}$.

    \begin{defi}\label{defi.zeros} 
    Let $\gamma\in\F_q^\times$ and $a\geq 1$ be an integer.
    For any place $v\in P_q(a)$, we denote by $\F_v$ the residue field of $K$ at $v$ and by $d_v := [\F_v:\F_q]$ the degree of $v$. 
    Viewing $v$ as the $\Gal(\bar{\F_q}/\F_q)$-orbit of an $\bar{\F_q}$-rational point of $\mathbb{G}_{m /\F_q}$, 
    we may pick an element $\beta_v\in\F_v^\times \subset \bar{\F_q}^\times$ representing that orbit~$v$. 
    The various choices  of representatives of $v$ in $\bar{\F_q}^\times$ are then $\beta_v, \beta_v^q, \beta_v^{q^2}, \dots, \beta_v^{q^{d_v-1}}$.
    The map $\psi_{\F_v}^{(\beta_v)}: x\mapsto \psi_q\circ\trace_{\F_v/\F_q}(\beta_v\cdot x)$ defines a non-trivial additive character on $\F_v$. 
    \begin{itemize}
    \item	
        We denote the Gauss sum $\gauss{\F_v}{\psi_{\F_v}^{(\beta_v)}}{\lambda_{\F_v}}$ by  $\gaun(v)$.
        A repeated application of \refGauss{item.gauss.orbit} shows that the definition of $\gaun(v)$ makes sense, 
        in that  the sum $\gauss{\F_v}{\psi_{\F_v}^{(\beta_v)}}{\lambda_{\F_v}}$ does not depend on a particular choice of representative $\beta_v$ for the  orbit $v$. 

    \item 
        We denote the Kloosterman sum $\kloos{\F_{v}}{\psi_{\F_v}^{(\beta_v)}}{\gamma}$ by $\Kloosn_\gamma(v)$, 
        and we let $\{\kln_\gamma(v), \kln'_\gamma(v)\}$ be the pair of algebraic integers attached to the Kloosterman sum $\Kloosn_\gamma(v)$ as in \refKloos{item.kloos.HD}.
        Again, these definitions make sense: repeated applications of \refKloos{item.kloos.orbit} imply that the value of  $\kloos{\F_{v}}{\psi_{\F_v}^{(\beta_v)}}{\gamma}$, and hence the pair of algebraic integers attached to it by \refKloos{item.kloos.HD}, do not depend on the choice of $\beta_v$ in $v$. 
    \end{itemize}
    \end{defi}

 %% -- %% -- %% -- %% -- %% -- %% -- %% -- %% -- %% -- %% -- %% -- %% -- %% -- %% -- %% -- %% -- %% -- %% -- %% -- %% -- %% -- %% -- %% -- %% -- %%
\section{The $L$-function of $E_{\gamma, a}$} 
\label{sec.Lfunc}

The main goal of this section is to provide an explicit expression for the $L$-function of the elliptic curve~$E_{\gamma, a}$, which will be instrumental for the proof of our main result. 
In the notation set up in the previous section, the result is as follows:

    \begin{theo}\label{theo.Lfunc} 
    Let $\F_q$ be a finite field of odd characteristic. 
    For any $\gamma\in\F_q^\times$ and any $a\geq 1$, consider the elliptic curve $E_{\gamma, a}$ over $K=\F_q(t)$ defined by \eqref{eq.Wmod}.
    The $L$-function $L(E_{\gamma, a}, T)\in\Z[T]$ of $E_{\gamma, a}$ is given by:
    \begin{equation}\label{eq.Lfunc}
    L(E_{\gamma, a}, T) = \prod_{v\in P_q(a)}\big(1-\gaun(v)\kln_\gamma(v)\cdot T^{\deg v}\big)\big(1-\gaun(v)\kln'_\gamma(v)\cdot T^{\deg v}\big),
    \end{equation}
    where $P_q(a)$ denotes the set defined in \S\ref{sec.places.a}, and 
    $\gaun(v)$, $\kln_\gamma(v)$, $\kln'_\gamma(v)$ denote the algebraic integers attached to any  $v\in P_q(a)$ in \S\ref{sec.gaun.kln}.
    \end{theo}

The proof of this theorem occupies the rest of the section: the next subsection proves a useful identity between character sums, and the following subsection contains the computation leading to Theorem \ref{theo.Lfunc}.

    \begin{rema}\label{rema.Lfunc}
    \begin{enumerate}[(1)]
    \item 
        Given the definition of $\kln_\gamma(v), \kln'_{\gamma}(v)$ and \refKloos{item.kloos.HD}, an equivalent way of formulating \eqref{eq.Lfunc} is:
        \begin{equation}\label{eq.Lfunc.alter}
        L(E_{\gamma, a}, T) 
        = \prod_{v\in P_q(a)}\big(1-\gaun(v)\Kloosn_\gamma(v)\cdot T^{\deg v}+\gaun(v)^2 q^{\deg v} \cdot T^{2\deg v}\big). 
        \end{equation}

    \item
        To define the sums $\gaun(v)$ and $\Kloosn_\gamma(v)$ for a place $v\in P_q(a)$, we started by fixing an additive character $\psi_q$ on $\F_q$.
        Note that the expression for the $L$-function of $E_{\gamma, a}$ obtained in Theorem \ref{theo.Lfunc} is independent of this choice.
        Indeed,  a different choice of $\psi_q$ has the sole effect of  permuting the factors in \eqref{eq.Lfunc}, for the $L$-function $L(E_{\gamma, a}, T)$ has integral coefficients.

    \item 
        By definition, we have $\sum_{v\in P_q(a)} \deg v = |\mathbb{G}_m(\F_{q^a})| = q^{a}-1$.
        Hence, as a polynomial in $T$, the $L$-function of $E_{\gamma, a}$  has degree $b(E_{\gamma, a}) = \deg L(E_{\gamma, a}, T) = 2(q^a-1)$. 
        This is compatible with the Grothendieck--Ogg--Shafarevich formula and the value of $\deg\cond (E_{\gamma, a})$ found in \eqref{eq.invariants}. 
        Note that $b(E_{\gamma, a})$ is even.
    \end{enumerate}	
    \end{rema}

\subsection{An identity between character sums}

Let $\F$ be a finite field of odd characteristic, equipped with an additive character $\psi$. 
Denote by $\lambda_\F=\lambda: \F^\times\to\{\pm1\}$ the unique quadratic character on $\F^\times$, extended by $\lambda(0) := 0$.
For any $\gamma\in\F^\times$, define the double character sum
\[S(\F, \psi, \gamma) 
:= \sum_{z\in\F}\sum_{x\in\F} \lambda(x^3 + z x^2+ \gamma x) \cdot\psi(z).\]
Note that the terms with $x=0$ do not contribute to the sum  since $\lambda(0)=0$, so that
\[S(\F, \psi, \gamma)
= \sum_{x\neq 0} \lambda(x^2) \left\{\sum_{z\in\F} \lambda\left(z + \frac{x^2+\gamma}{x}\right)\cdot\psi(z)\right\}.\]
For any $x\in\F^\times$, we put $u=z + (x^2+\gamma)/x$ in the sum displayed between brackets: we obtain that
\[\sum_{z\in\F} \lambda\left(z + \frac{x^2+\gamma}{x}\right)\cdot\psi(z)
= \psi\left(-x -\frac{\gamma}{x}\right)\cdot\sum_{u\in\F}\lambda(u)\cdot\psi(u)
= - \psi\left(-x -\frac{\gamma}{x}\right) \cdot \gauss{\F}{\psi}{\lambda}.\]
Summing this identity over all $x\in\F^\times$ then yields that
\[S(\F, \psi, \gamma)
= - \gauss{\F}{\psi}{\lambda} \cdot \sum_{x\in\F^\times} \lambda(x)^2 \cdot \psi\left(-x -\frac{\gamma}{x}\right) 
= - \gauss{\F}{\psi}{\lambda} \cdot \sum_{y\in\F^\times}  \psi\left(y +\frac{\gamma}{y}\right),\]
where we have put $y=-x$ (note that $\lambda(x)^2=1$ for all $x\neq 0$ because $\lambda$ has order $2$). 
We have thus proved:
    \begin{lemm}\label{lemm.charsum.identity}
    In the above setting, one has
    \[ S(\F, \psi, \gamma) 
    = \gauss{\F}{\psi}{\lambda} \cdot \kloos{\F}{\psi}{\gamma},\]	
    where the Gauss and Kloosterman sums are as defined in \S\ref{sec.gauss.kloos}.
    \end{lemm}
    
It is immediate to check that all three sums in the above identity vanish when $\psi$ is trivial.

\subsection{Proof of Theorem \ref{theo.Lfunc}}

Fix a parameter $\gamma\in\F_q^\times$ and an integer $a\geq 1$ as in the statement of the theorem.
Our starting point for the computation is  expression \eqref{eq.expr.Lfunc} below for the $L$-function of $E_{\gamma, a}$. 

For any $\tau\in\bar{\F_q}$, we denote by $v_\tau$ the corresponding finite place of $K$;
we pick a minimal integral Weierstrass (affine) model of $E_{\gamma, a}/K$ at $v_\tau$ of the form $y^2=f_{\tau}(x,t)$ where $f_\tau(x,t)$ is a monic cubic polynomial in $x$ with coefficients in $\F_q[t]$.
Recall that $\lambda_{\F_{q^n}} : \F_{q^n}\to\{\pm 1\}$ denotes the unique  character of exact order $2$ on $\F_{q^n}^\times$, extended by $\lambda_{\F_{q^n}}(0):=0$ to the whole of $\F_{q^n}$.
With this notation, one has the equality of formal power series:
\begin{equation}\label{eq.expr.Lfunc}
\log L(E_{\gamma, a}, T)
= - \sum_{n=1}^\infty\left(\sum_{\tau\in\F_{q^n}} \sum_{x\in\F_{q^n}} \lambda_{\F_{q^n}}\big(f_{\tau}(x, \tau)\big)\right)\cdot \frac{T^n}{n}.
\end{equation}
This expression can be derived from the definition \eqref{eq.defi.Lfunc.euler}-\eqref{eq.defi.Lfunc} of the $L$-function, just as in \cite{Griffon_LegAS} (see Lemma~4.6 and the following paragraph there) or \cite[\S4]{GriffonUlmer}.
Here, we have implicitly used the fact that~$E_{\gamma, a}$ has additive reduction at the   place $\infty$  (see Proposition \ref{prop.bad.red}) to ignore the local terms corresponding to this place.
At a place of additive reduction, the local Euler factor of $L(E_{\gamma, a}, T)$ in \eqref{eq.defi.Lfunc}    is indeed trivial.

We next aim at giving a more explicit expression of the inner double sums in \eqref{eq.expr.Lfunc}.
As was pointed out in Remark \ref{rema.mini.mod}, one can choose $f_\tau(x, t) =  x^3+ \wp_a(t)\cdot x^2 + \gamma x$ for any $\tau\in\bar{\F_q}$. 
For any integer $n\geq 1$, the inner double sum in \eqref{eq.expr.Lfunc} can thus be rewritten as
\[ \sum_{\tau\in\F_{q^n}} \sum_{x\in\F_{q^n}} \lambda_{\F_{q^n}}\big(f_{\tau}(x, \tau)\big)
= \sum_{z\in\F_{q^n}}\sum_{x\in\F_{q^n}} \big|\big\{\tau\in\F_{q^n} : \wp_a(\tau) = z \big\}\big| \cdot \lambda_{\F_{q^n}}\big(x^3 + z\cdot x^2 + \gamma x\big). \]
Moreover, we know from \cite[Lemma 4.5]{Griffon_LegAS} that, for any $z\in\F_{q^n}$, 
\[ \big|\big\{\tau\in\F_{q^n} : \wp_a(\tau) = z \big\}\big| 
= \sum_{\beta\in\F_{q^n}\cap\F_{q^a}} \psi_{q}\circ\trace_{\F_{q^n}/\F_q}(\beta\cdot x),\]
where, for all $\beta\in\F_{q^n}$, the map $z\mapsto \psi_{q}\circ\trace_{\F_{q^n}/\F_q}(\beta\cdot z)$ is an additive character on $\F_{q^n}$, which we denote by $\psi_{q^n}^{(\beta)}$.
Hence, for any integer $n\geq 1$, we have
\begin{multline*}
\sum_{\tau\in\F_{q^n}}\sum_{x\in\F_{q^n}} \lambda_{\F_{q^n}}\big(f_{\tau}(x)\big)
=\sum_{\beta\in\F_{q^n}\cap\F_{q^a}} S(\F_{q^n}, \psi_{q^n}^{(\beta)}, \gamma), \\
\text{where }
S(\F_{q^n},\psi_{q^n}^{(\beta)},\gamma)
=\sum_{z\in\F_{q^n}}\sum_{x\in\F_{q^n}} \lambda_{\F_{q^n}}(x^3 + z x^2+ \gamma x)\cdot\psi_{q^n}^{(\beta)}(z).
\end{multline*}
Lemma \ref{lemm.charsum.identity}  now yields that 
$S(\F_{q^n}, \psi_{q^n}^{(\beta)}, \gamma) = \gauss{\F_{q^n}}{\psi_{q^n}^{(\beta)}}{\lambda} \cdot \kloos{\F_{q^n}}{\psi_{q^n}^{(\beta)}}{\gamma}$ for all $\beta\in\F_{q^n}$,
where the Gauss and Kloosterman sums are as in \S\ref{sec.gauss.kloos}.
Combining the above equalities, we obtain that
\begin{equation}\label{eq.Lfunc.inter1}
	-\log L(E_{\gamma, a}, T) 
	= \sum_{n=1}^\infty\left( \sum_{\beta\in\F_{q^n}\cap\F_{q^a}} \gauss{\F_{q^n}}{\psi_{q^n}^{(\beta)}}{\lambda} \cdot \kloos{\F_{q^n}}{\psi_{q^n}^{(\beta)}}{\gamma}	\right)  \cdot\frac{T^n}{n}.
\end{equation}
When $\beta = 0$, the character $\psi_{q^n}^{(\beta)}$ is trivial and the product $\gauss{\F_{q^n}}{\psi_{q^n}^{(\beta)}}{\lambda} \cdot \kloos{\F_{q^n}}{\psi_{q^n}^{(\beta)}}{\gamma}$ vanishes.
For any $\beta\in(\F_{q^a}\cap\F_{q^n})\smallsetminus\{0\}$, denote by $v_\beta$ the place of $K$ containing $\beta$ (equivalently, $v_\beta$ is the $\Gal(\bar{\F_q}/\F_q)$-orbit of $\beta$).
The place $v_\beta$  has degree $\deg v_\beta = [\F_q(\beta):\F_q]$, which divides both $a$ and $n$. 
In particular, $v_\beta$ belongs to $P_q(a)$.

	\begin{lemm}\label{lemm.charsum.def}
	In the notation of \S\ref{sec.gaun.kln}, one has
	\[\gauss{\F_{q^n}}{\psi_{q^n}^{(\beta)}}{\lambda} 
	= \gaun(v_{\beta})^{n/\deg v_\beta} \quad \text{ and }\quad
	\kloos{\F_{q^n}}{\psi_{q^n}^{(\beta)}}{\gamma} 
	= \kln_\gamma(v_{\beta})^{n/\deg v_\beta} + \kln'_\gamma(v_{\beta})^{n/\deg v_\beta}.\]
	\end{lemm}
	
	\begin{proof}
	For brevity, we write  $d= \deg v_{\beta}$ (recall that $d$ divides $n$). 
	By multiplicativity of the trace in towers of extensions and because $\beta\in\F_{q^d}$, we have $\psi_{\F_{q^n}}^{(\beta)} = \psi_{\F_{q^d}}^{(\beta)} \circ \trace_{\F_{q^n}/\F_{q^d}}$.
	Moreover, $\lambda = \lambda_{\F_{q^n}}$ coincides with $\lambda_{\F_{q^d}}\circ\norm_{\F_{q^n}/\F_{q^d}}$.
	Hence, the Hasse--Davenport relation for Gauss sums \refGauss{item.gauss.HD} implies that
	\[\gauss{\F_{q^n}}{\psi_{q^n}^{(\beta)}}{\lambda} 
	= \gauss{\F_{q^n}}{\psi_{q^d}^{(\beta)}\circ\trace_{\F_{q^n}/\F_{q^d}}}{\lambda_{\F_{q^d}}\circ\norm_{\F_{q^n}/\F_{q^d}}}
	= \gauss{\F_{q^d}}{\psi_{q^d}^{(\beta)}}{\lambda_{\F_{q^d}}}^{[\F_{q^n}:\F_{q^d}]}.\]
	Since, by definition, we have $\gaun(v_{\beta})=\gauss{\F_{q^d}}{\psi_{q^d}^{(\beta)}}{\lambda_{\F_{q^d}}}$, the first identity is proved.
	The second identity is proved in a similar fashion, using the Hasse--Davenport relation for Kloosterman sums \refKloos{item.kloos.HD}.
	\end{proof}

%\noindent
Plugging the identities of Lemma \ref{lemm.charsum.def} into \eqref{eq.Lfunc.inter1} and exchanging the order of summation yields that
\[ - \log L(E_{\gamma, a}, T) 
= \sum_{\beta\in\F_{q^a}\smallsetminus\{0\}} \sum_{ \substack{n\geq 1\\ \deg v_\beta \mid n}}\left( \big( \gaun(v_{\beta}) \kln_\gamma(v_\beta)\big)^{n/\deg v_\beta} + \big( \gaun(v_{\beta}) \kln'_\gamma(v_\beta)\big)^{n/\deg v_\beta}\right)  \cdot\frac{T^n}{n}. \]
Upon reindexing the second sum (by setting $m = n/\deg v_\beta$), we then obtain that
\begin{align*}
- \log L(E_{\gamma, a}, T) 
&=\sum_{\beta\in\F_{q^a}^\times}  \frac{1}{\deg v_\beta} \left\{ 
\sum_{m=1}^\infty \frac{(\gaun(v_{\beta}) \kln_\gamma(v_\beta) \cdot T^{\deg v_\beta})^m}{m} + \sum_{m=1}^\infty\frac{(\gaun(v_{\beta}) \kln'_\gamma(v_\beta) \cdot T^{\deg v_\beta})^m}{m }\right\} \\
&= - \sum_{\beta\in\F_{q^a}^\times}  \frac{1}{\deg v_\beta} \cdot \log\left((1-\gaun(v_{\beta}) \kln_\gamma(v_\beta) \cdot T^{\deg v_\beta})(1-\gaun(v_{\beta}) \kln'_\gamma(v_\beta) \cdot T^{\deg v_\beta})\right). 
\end{align*}
We finally group the indices $\beta\in\F_{q^a}^\times$ corresponding to the same place $v\in P_q(a)$, and get:
\begin{align*}
\log L(E_{\gamma, a}, T) 
& = \sum_{v\in P_q(a)}  \sum_{\beta\in v} \frac{1}{\deg v} \cdot  \log\left((1-\gaun(v) \kln_\gamma(v) \cdot T^{\deg v})(1-\gaun(v) \kln'_\gamma(v) \cdot T^{\deg v})\right).  \\
& = \sum_{v\in P_q(a)}  \log\left((1-\gaun(v) \kln_\gamma(v) \cdot T^{\deg v})(1-\gaun(v) \kln'_\gamma(v) \cdot T^{\deg v})\right).  
\end{align*}
To conclude the proof of Theorem \ref{theo.Lfunc}, there only remains to exponentiate this identity. 	\hfill$\Box$

 %% -- %% -- %% -- %% -- %% -- %% -- %% -- %% -- %% -- %% -- %% -- %% -- %% -- %% -- %% -- %% -- %% -- %% -- %% -- %% -- %% -- %% -- %% -- %% -- %%
\section{Non-vanishing of $L(E_{\gamma, a}, T)$ at the central point and consequences}

As before, let $\F_q$ be a finite field of odd characteristic $p$ and $K:=\F_q(t)$.
For any parameter $\gamma\in\F_q^\times$ and any integer $a\geq 1$, we consider the elliptic curve $E_{\gamma, a}$ defined over $K$ by \eqref{eq.Wmod}.
In this section, we describe the behaviour of the $L$-function $z\mapsto L(E_{\gamma, a}, z)$ around $z=1$ (see Theorem \ref{theo.nonvanishing.Lfunc}).
We first gather some information about its $p$-adic slope sequence from Theorem \ref{theo.Lfunc} and results about Gauss and Kloosterman sums. 
We then derive an important corollary of this analytic study, namely the BSD conjecture for the elliptic curves $E_{\gamma, a}$ (see Corollary \ref{coro.BSD}).

\subsection{$p$-adic slopes of $L(E_{\gamma, a}, T)$}\label{sec.padic.slopes}

We choose, once and for all, a prime ideal $\gP$ of $\Qbar$ above $p$.
We denote by  $\ord_\gP:\Qbar^\times\to\Q$ the corresponding discrete valuation, so normalised that $\ord_{\gP}(q)=1$. 
The goal of this subsection is to compute the $p$-adic slope sequence of $L(E_{\gamma, a}, T)$ explicitly (see \S\ref{sec.def.Lfunc}).

We first prove a probably well-known lemma for which we could not find a proof in the literature:
	\begin{lemm}\label{lemm.kloos.slope} 
	Let $\F$ be a finite extension of $\F_q$,  and $\psi$ be a nontrivial additive character on $\F$. 
	For any $\alpha\in\F^\times$,  consider the Kloosterman sum $\kloos{\F}{\psi}{\alpha}$ and the pair $\{\klo_\F(\psi;\alpha), \klo'_\F(\psi;\alpha)\}$ of algebraic integers associated to it by \refKloos{item.kloos.HD}. 
	Then we have
	\begin{equation*}
		\{\ord_{\gP}\klo_\F(\psi;\alpha), \ord_{\gP}\klo'_\F(\psi;\alpha)\} 
		= \big\{ 0, [\F:\F_q] \big\}.
	\end{equation*}
	Equivalently, one has $\ord_\gP\kloos{\F}{\psi}{\alpha}=0$.
	\end{lemm}

\noindent
The last statement is equivalent to the first assertion of \refKloos{item.kloos.avoid}. 
In the same setting,  the $\gP$-adic valuation of the  Gauss sum $\gauss{\F}{\psi}{\lambda}$  is easily determined: %computed: 
\refGauss{item.gauss.quad} shows that the sums $\gauss{\F}{\psi}{\lambda}$ has the same $\gP$-adic valuation as $|\F|^{1/2}$, so that 
\begin{equation}\label{eq.Gauss.slope}
\ord_{\gP} \gauss{\F}{\psi}{\lambda} 
= \frac{[\F:\F_q]}{2}.	
\end{equation}

	\begin{proof}
	By construction, $\kloos{\F}{\psi}{\alpha}$ is an element of $\Q(\zeta_p)$.
	The unique prime ideal of $\Q(\zeta_p)$ above $p$ is the principal ideal $I = (\zeta_p-1)$, so that $p\cdot\Z[\zeta_p]= I^{p-1}$.
	Since $\psi(y)$ is a power of $\zeta_p$ for all $y\in\F$, we have $\psi(y)\equiv 1 \bmod I$ and we find that 
	\[-\kloos{\F}{\psi}{\alpha} 
	\equiv \sum_{x\in\F^\times} 1 
	\equiv |\F|-1 \equiv  -1\bmod I,\]
	because $|\F|$ is a power of $p$.
	In particular, $\kloos{\F}{\psi}{\alpha}\not\equiv 0\bmod p$ in $\Qbar$, and we have ${\ord_\gP \kloos{\F}{\psi}{\alpha} = 0}$. 
	
	On the other hand, \refKloos{item.kloos.HD} implies that the algebraic integers $\klo_\F(\psi;\alpha)$ and $\klo'_\F(\psi;\alpha)$ satisfy
	\[ \klo_\F(\psi;\alpha) \cdot \klo'_\F(\psi;\alpha) = |\F|  \quad\text{ and }\quad
	\klo_\F(\psi;\alpha) + \klo'_\F(\psi;\alpha) = \kloos{\F}{\psi}{\alpha}. \]
	Hence, the pair $\{v,v'\}$ formed by their $\gP$-adic valuations satisfies: 
	$v, v'\geq 0$, 
	$v+v' = \ord_{\gP}|\F| = [\F:\F_q]$, and 
	$\min\{v,v'\}\leq \ord_\gP \kloos{\F}{\psi}{\alpha} = 0$.
	The only  pair $\{v,v'\}$ fulfilling these requirements is $\{0, [\F:\F_q]\}$. Hence the result.
	\end{proof}

We can now state and prove the following:

	\begin{theo}\label{theo.Lfunc.slopes}
	Let $\gamma \in\F_q^\times$ and $a\geq 1$. 
	Consider the elliptic curve $E_{\gamma, a}/K$ defined by \eqref{eq.Wmod}, and write $b=\deg L(E_{\gamma, a}, T)$.
		Recall from Remark \ref{rema.Lfunc}(3) that $b=2(q^a-1)$ is even.
	The $L$-function $L(E_{\gamma, a}, T)$ has $p$-adic slope sequence
	\[ \lambda_1=\frac 1 2, \ 
	\lambda_2=\frac 1 2, \ \dots, \ 
	\lambda_{b/2}=\frac 1 2,\ 
	\lambda_{b/2+1}=\frac 3 2,\ 
	\lambda_{b/2+2}=\frac 3 2,\ \dots,\ 
	\lambda_b=\frac 3 2.\]
	\end{theo}
	
	\begin{proof} 
	By definition (see \S\ref{sec.def.Lfunc}), the $p$-adic slope sequence of $L(E_{\gamma, a}, T)$ is the (suitably indexed) multiset of $\gP$-adic valuations of the inverses of zeros of $z\mapsto L(E_{\gamma, a}, z)$.
	
	Upon staring at the expression for $L(E_{\gamma, a}, T)$ obtained in Theorem \ref{theo.Lfunc}, one immediately sees that an algebraic number $z\in\Qbar$ is the inverse of a zero of $L(E_{\gamma, a}, T)$  
	if and only if there exists a place $v\in P_q(a)$ with $z^{\deg v} \in \{\gaun(v)\kln_\gamma(v), \gaun(v)\kln'_\gamma(v) \}$.
	In particular, if $z$ is the inverse of a zero of $L(E_{\gamma, a}, T)$, there exists $v\in P_q(a)$ such that $\ord_{\gP} (z^{\deg v})$ equals one of $\ord_\gP(\gaun(v)\kln_\gamma(v) )$ or $\ord_\gP(\gaun(v)\kln'_\gamma(v))$.
	
	Applying \eqref{eq.Gauss.slope} in the case where $\gauss{\F}{\psi}{\lambda}=\gaun(v)$, we  see that $\ord_{\gP} \gaun(v) = (\deg v)/2$.
	Besides, the previous lemma applied to $\kloos{\F}{\psi}{\alpha} = \Kloosn_\gamma(v)$ yields that 
	$\big\{\ord_\gP\kln_\gamma(v), \ord_\gP \kln'_\gamma(v)  \big\} 	= \{0, \deg v\}$.
	We therefore have
	\begin{equation*}
	\ord_\gP z  
	= \frac{\ord_{\gP}(z^{\deg v})}{\deg v } 
	\in \left\{ \frac{\ord_{\gP} \gaun(v) + \ord_{\gP} \kln_\gamma(v)}{\deg v}, \frac{\ord_{\gP} \gaun(v) + \ord_{\gP} \kln'_\gamma(v)}{\deg v}\right\} %= \left\{ \frac{1}{2} +0, \frac12 +1\right\} 
	= \left\{ \frac 12, \frac 32\right\}.
	\end{equation*}
	
	As was recalled in \S\ref{sec.def.Lfunc}, the $p$-adic slope sequence $\{\lambda_k\}_{k=1}^{b}$ of  $L(E_{\gamma, a}, T)$ admits the following symmetry: for a given $s\in[0,2]$, there are as many indices $k$ such that $\lambda_k=s$ as there are indices $k'$ such that $\lambda_{k'}=2-s$.
	From this and from the above display, we immediately conclude that, among the $b$ slopes of the $L$-function~$L(E_{\gamma, a}, T)$, half of them equal $1/2$ and the other half equal $3/2$.
	The result now follows upon choosing a suitable numbering of these slopes. 
	\end{proof}

\subsection{Non-vanishing at the central point}
The following follows almost immediately from Theorem~\ref{theo.Lfunc.slopes}:

	\begin{theo}\label{theo.nonvanishing.Lfunc} 
	For any $\gamma\in\F_q^\times$ and any $a\geq 1$, the $L$-function $L(E_{\gamma, a}, T)$ does not vanish at $T=q^{-1}$. %one has  $L(E_{\gamma, a}, q^{-1}) \neq 0$. 
	In other words, one has $\ord_{T=q^{-1}}L(E_{\gamma,a}, T)=0$.
	\end{theo}
	\begin{proof} 
	Given the expression for $L(E_{\gamma, a}, T)$ of Theorem \ref{theo.Lfunc}, it suffices to show that, for any $v\in P_q(a)$, the two factors ${1-\gaun(v)\kln_\gamma(v) q^{-\deg v}}$ and ${1-\gaun(v)\kln'_\gamma(v) q^{-\deg v}}$ of $L(E_{\gamma, a}, q^{-1})$ are nonzero.
	As we have seen in the previous subsection, for all $v\in P_q(a)$, we have
	\[\big\{\ord_\gP(\gaun(v)\kln_\gamma(v) ), \ord_\gP (\gaun(v)\kln'_\gamma(v) )\big\} 
	= \left\{ \frac{\deg v}{2}, \frac{3 \deg v}{2}\right\}.\]
	Since $\ord_\gP(q^{\deg v}) = \deg v$, neither of $\gaun(v)\kln_\gamma(v)$ or $\gaun(v)\kln'_\gamma(v)$  can equal $q^{\deg v}$, and we are done.
	\end{proof}

Remark \ref{rema.nonvanishing.reloaded} outlines an alternative proof of Theorem \ref{theo.nonvanishing.Lfunc}: there, we obtain the non-vanishing by  relying on  the ``angular distribution'' of the Gauss and Kloosterman sums.

\subsection{The BSD conjecture for $E_{\gamma, a}$}

Given the non-vanishing of $L(E_{\gamma, a}, T)$ at $T=q^{-1}$, we deduce from the BSD result (see Theorem \ref{conj.BSD}) that  the BSD conjecture holds for all the elliptic curves $E_{\gamma, a}$. 
	\begin{coro}\label{coro.BSD}
	For all $\gamma\in\F_q^\times$ and all integers $a\geq 1$, one has:
	\begin{enumerate}[(1)]
		\item\label{coro.rk.0}
		The Mordell--Weil group $E_{\gamma, a}(K)$ consists of the  two points 
		$\O$ and $P_0=(0,0)$.
		\item\label{coro.finite.sha} 
		The Tate--Shafarevich group $\sha(E_{\gamma, a})$ is finite.
		\item 
		The following identity holds:
		\begin{equation}\label{eq.BSD.2}
			L(E_{\gamma, a}, q^{-1}) = \frac{|\sha(E_{\gamma, a})|}{q^{-1}\cdot H(E_{\gamma, a})}.
		\end{equation}
	\end{enumerate}
	\end{coro}

That the elliptic curves $E_{\gamma, a}$ satisfy the BSD conjecture is not new: as was already mentioned, it was first proved by Pries and Ulmer in \cite{UlmerPries}, where the proof relies on the geometry of the minimal regular model of $E_{\gamma, a}$. 
Our proof appears to be largely independent of theirs. 
Pries and Ulmer have further shown (also by a geometric argument) that $E_{\gamma, a}(K)$ has rank $0$ (see \S6.2 in \cite{UlmerPries}).

	\begin{proof} 
	By Theorem \ref{conj.BSD}{\it(1)}, the Mordell--Weil group $E_{\gamma, a}(K)$ is finite, hence coincides with its torsion subgroup.
	Theorem \ref{theo.tors} then yields the first assertion.
	The second assertion is copied verbatim from the second item of Theorem \ref{conj.BSD}.
	As for identity \eqref{eq.BSD.2}, it suffices to plug the values $|E_{\gamma, a}(K)|=2$ and $\tam(E_{\gamma, a})=4$ (calculated in \S\ref{sec.invariants}) into the BSD formula \eqref{eq.conj.BSD}.
	\end{proof}

 %% -- %% -- %% -- %% -- %% -- %% -- %% -- %% -- %% -- %% -- %% -- %% -- %% -- %% -- %% -- %% -- %% -- %% -- %% -- %% -- %% -- %% -- %% -- %% -- %% -- %%
\section[Distribution of angles]{Distribution of the sums $\Kloosn_\gamma(v)$}
\label{sec.distribution}

In this section, we work in the following setting: 
we fix a finite field $\F_q$ of odd characteristic $p$, endowed with a nontrivial additive character $\psi_q$; we also choose, once and for all, an embedding $\iota:\Qbar\into\C$.
For any $\gamma\in\F_q^\times$ and any $a\geq 1$, we consider the elliptic curves $E_{\gamma, a}$ defined by \eqref{eq.Wmod}. 

In the next section, we will describe the asymptotic behaviour of $L(E_{\gamma, a}, q^{-1})$ as $a\to\infty$. 
Doing so will require some knowledge about the ``angular distribution'' in the complex plane of (the images under~$\iota$ of) the algebraic integers $\kln_\gamma(v), \kln'_\gamma(v)$, for $v\in P_q(a)$.  
We therefore lay out, in this section, the relevant facts about that distribution.

\subsection{Angles of the sums}\label{sec.angles}

For a place $v\in P_q(a)$, the images under $\iota$ of the algebraic integers $\gaun(v)$, $\kln_\gamma(v)$ and $\kln'_\gamma(v)$ lie on the complex circle 
$\{z\in\C : |z| = q^{\deg v/2}\}$ by the Weil bounds \refGauss{item.gauss.magnitude} and \refKloos{item.kloos.magnitude} for Gauss and Kloosterman sums.
We can therefore introduce the following angles:

	\begin{defi}\label{defi.angles}
	For a place $v\in P_q(a)$, 
	\begin{itemize}
	\item 
	We know from \refGauss{item.gauss.quad} that the algebraic number $\gaun(v)/q^{\deg v/2}$ is a $4$th root of unity. 
	Thus, there is a unique $\gauan(v)\in\{0, \pi/2, \pi, 3\pi/2\}$ such that
	\[\iota(\gaun(v)) 
	= q^{\deg v/2}\cdot \e^{i\gauan(v)}.\]

	\item 
	By 	\refKloos{item.kloos.magnitude}, the complex number 	$\iota(\Kloosn_\gamma(v))$ is  real with $|\iota(\Kloosn_\gamma(v))|\leq 2q^{\deg v/2}$.
	Hence there exists a unique angle $\kloan_\gamma(v)\in[0,\pi]$ such that 
	\[\iota(\Kloosn_\gamma(v)) 
	= 2q^{\deg v/2}\cdot\cos\kloan_\gamma(v).\]
	\end{itemize}
	\end{defi}

For all $v\in P_q(a)$, we remark that one has 
$\big\{\iota(\kln_\gamma(v)), \iota(\kln'_\gamma(v)) \big\} 
= \left\{q^{\deg v/2} \cdot \e^{i\kloan_\gamma(v)}, q^{\deg v/2} \cdot \e^{- i\kloan_\gamma(v)}\right\}$.
With this new notation, the expression for $L(E_{\gamma, a}, T)$ obtained in Theorem \ref{theo.Lfunc} becomes:
\begin{equation}\label{eq.Lfunc.alter2}
L(E_{\gamma, a}, T) 
= \prod_{v\in P_q(a)} \left(1 - \e^{i(\gauan(v) + \kloan_\gamma(v))}\cdot(qT)^{\deg v}\right)\left(1 - \e^{i(\gauan(v) - \kloan_\gamma(v))}\cdot(qT)^{\deg v}\right).	
\end{equation}
Note that, even though the angles $\kloan_\gamma(v)$ individually depend on the choice of $\iota$, the set $\{\kloan_\gamma(v)\}_{v\in P_q(a)}$ does not: 
a different choice of $\iota$ only permutes the various $\kloan_\gamma(v)$, since $L(E_{\gamma, a}, T)$ has integral coefficients.

Evaluating both sides of the above equality at $T=q^{-1}$ yields
\begin{equation}\label{eq.spval.expr}
L(E_{\gamma, a}, q^{-1}) 
= \prod_{v\in P_q(a)} \left(1 + \e^{2i\gauan(v)} -2 \e^{i\gauan(v)}\cdot \cos\kloan_\gamma(v)\right).
\end{equation}
%Combined with the BSD formula \eqref{eq.BSD.2} and \eqref{eq.invariants}, this yields a closed expression for the order of~$\sha(E_{\gamma, a})$:
%\[|\sha(E_{\gamma, a})| = \prod_{v\in P_q(a)} q^{\deg v/2}\left(1 + \e^{2i\gauan(v)} -2 \e^{i\gauan(v)}\cdot \cos\kloan_\gamma(v)\right).\]	

	\begin{rema}\label{rema.nonvanishing.reloaded}
	\begin{enumerate}[(1)]
	\item 
		For any $v\in P_q(a)$, Lemma \ref{lemm.kloos.slope} implies that the angle $\kloan_\gamma(v)$ does not lie in $\{0, \pi/2, \pi\}$.
		Indeed, were $\kloan_\gamma(v)$ to hit one of $0, \pi/2$ or $\pi$, the sum $\Kloosn_\gamma(v)$ would equal $+2q^{\deg v/2}, 0$ or $-2q^{\deg v/2}$, respectively. 
		This is incompatible with Lemma \ref{lemm.kloos.slope} since $\Kloosn_\gamma(v)$ would then not be a $p$-adic unit. 
	
		More generally, with the same construction as in Definition \ref{defi.angles}, one can associate an angle $\theta_{\F, \psi, \alpha}\in[0, \pi]$ to any of the Kloosterman sums $\kloos{\F}{\psi}{\alpha}$ introduced in \S\ref{sec.gauss.kloos}.
		A similar argument as the above  then shows that $\theta_{\F, \psi, \alpha}\notin \{0, \pi/2, \pi\}$. 
		This directly implies the second assertion in \refKloos{item.kloos.avoid}.
	
	\item
		The previous item actually provides a second proof of the non-vanishing of the $L$-function $L(E_{\gamma, a}, T)$ at~$T=q^{-1}$, as follows.
		It is clear that, for a place $v\in P_q(a)$, the factor indexed by $v$ in the product~\eqref{eq.Lfunc.alter2} vanishes at $T=q^{-1}$ if and only if $\gauan(v) \equiv \pm\kloan_\gamma(v)\bmod{2\pi}$.
		By Definition \ref{defi.angles} above, we know that $\gauan(v)\in\{0, \pi/2, 3\pi/2, 2\pi\}$. 
		What we have shown in item (1)  proves that the $v$-th factor in \eqref{eq.Lfunc.alter2} does not vanish at $T=q^{-1}$. 
		Hence $L(E_{\gamma, a}, q^{-1})$ is nonzero.
	\end{enumerate}
	\end{rema}

\subsection[Distribution of Kloosterman angles]{Angular distribution of the sums $\Kloosn_\gamma(v)$}
\label{sec.distrib.kloos}

Let $\mu_\infty$ denote the Sato--Tate measure on~$[0,\pi]$: recall that $\dd\mu_{\infty} := \frac{2}{\pi}\sin^2\theta\dd\theta$.
Consider the set $\{\kloan_\gamma(v)\}_{v\in P_q(a)}$ of angles introduced in the previous subsection. 
We denote by $\mu_a$ the discrete probability measure on $[0,\pi]$ supported on $\{\kloan_\gamma(v)\}_{v\in P_q(a)}$. 
In other words, we put
\[\mu_a := \frac{1}{|P_q(a)|} \sum_{v\in P_q(a)}\delta\{\kloan_\gamma(v)\},\]
where $\delta\{x\}$ denotes the Dirac delta measure at $x\in[0,\pi]$. 
Note that $\mu_a$    depends neither on the choice of a non-trivial additive character $\psi_q$ on $\F_q$ nor on that of an embedding $\iota:\Qbar\into\C$. 
The measure $\mu_a$ does depend on $\F_q$ and $\gamma$ though, % our choices of $\F_q$, $\psi_q$, $\iota$ and $\gamma$, 
but  we chose not to reflect this in the notation for brevity. 

In a previous work, the first-named author has proved that, as $a\to\infty$,  the sequence of measures $\{\mu_a\}_{a\geq 1}$ converges weak-$\ast$ to $\mu_\infty$ in a quantitative way (see Theorem 6.6 in \cite{Griffon_LegAS}). 
More specifically, we have

	\begin{theo}\label{theo.distrib.kloos} 
	In the above setting, given any continuously differentiable function $g:[0,\pi]\to\C$ and any~$\gamma\in\F_q^\times$,  
	as $a\geq 1$ tends to $+\infty$, one has
	\begin{equation}
	\left|\int_{[0,\pi]} g \dd\mu_a - \int_{[0,\pi]} g \dd\mu_\infty\right|
	\ll_{q} \frac{a^{1/2}}{q^{a/4}}\cdot\int_{0}^\pi |g'(t)|\dd t,
	\end{equation}
	where the implicit constant is effective and depends at most on $q$.
	\end{theo}

\noindent
The reader is referred to section 6 of \cite{Griffon_LegAS} for a detailed proof of that statement. 
Note that $\mu_a$ is the same measure as the one denoted by $\nu_a$ there.

\subsection{`Small' angles of Kloosterman sums} 
\label{sec.small.angles}
As we have seen in Remark \ref{rema.nonvanishing.reloaded}, the fact that $L(E_{\gamma, a}, T)$ does not vanish at $T=q^{-1}$ is equivalent to the statement that none of the angles $\kloan_\gamma(v)$ lies in $\{0, \pi/2, \pi\}$.  
In this subsection, we obtain a quantitative version of  the fact that the set $\{\kloan_\gamma(v)\}_{v\in P_q(a)}$ ``avoids'' $\{0, \pi/2, \pi\}$.

	\begin{theo}\label{theo.small.angles}
	There exists a positive constant $\sigma_p$, depending at most on $p$, such that the following holds.
	Let $\F_q$ be a finite field of characteristic $p$ endowed with a nontrivial additive character $\psi_q$, and fix an embedding $\iota:\Qbar\into\C$.
	For any parameter $\gamma\in\F_q^\times$ and any integer $a\geq 1$, one has 
	\[\{\kloan_\gamma(v)\}_{v\in P_q(a)}
	\subset \big[\epsilon_{a}, \tfrac{\pi}{2} - \epsilon_{a}\big]
	\cup\big[\tfrac{\pi}{2} + \epsilon_{a}, \pi - \epsilon_{a}\big]
	\quad \text{where } \epsilon_a = (q^a)^{-\sigma_p}.\]
	In other words, the measure $\mu_a$ is supported on $\big[\epsilon_{a}, \tfrac{\pi}{2} - \epsilon_{a}\big]\cup\big[\tfrac{\pi}{2} + \epsilon_{a}, \pi - \epsilon_{a}\big]$. 
	\end{theo}

\noindent
Corollary 5.5 in \cite{Griffon_LegAS} already shows that $\kloan_\gamma(v)\in\big[\epsilon_a, \pi-\epsilon_a]$ for all $v\in P_q(a)$;
so we actually only need to prove that $|\kloan_\gamma(v)-\pi/2|\geq \epsilon_a$.
However, for   convenience, %for the convenience of the reader, 
we give a complete proof of Theorem \ref{theo.small.angles}.

The main tool to prove this result is the following version of Liouville's inequality,  which is taken from the introduction of \cite{MiWa}.  
Let $P\in\Z[X]$ be a polynomial of degree $N$. For any algebraic number $\kappa\in\Qbar$, let $ht(\kappa)$ denote its absolute logarithmic Weil height. % (see \cite{MiWa} for the definition). 
If $P(\kappa)\neq 0$, then
\begin{equation}\label{eq.liouville}
\frac{\log |P(\kappa)|}{[\Q(\kappa):\Q]}\geq -\log\|P\|_1 -N \cdot ht(\kappa),
\end{equation}
in any complex embedding of $\Q(\kappa)$.
Here we have denoted by $\|P\|_1$  the sum of the absolute values of the coefficients of $P$.
The reader is referred to \cite{MiWa} for this statement and its proof.

	\begin{proof}[Proof of Theorem \ref{theo.small.angles}] 
	In the setting of the theorem, we claim that there exists a constant $\sigma_p>0$ such that, for all $v\in P_q(a)$, one has
	\[\min\{\kloan_\gamma(v), |\kloan_\gamma(v) - \pi/2|, \pi-\kloan_\gamma(v)\}\geq (q^{\deg v})^{-\sigma_p}. \]
	Since $\deg v$ divides $a$ for all $v\in P_q(a)$, the theorem clearly follows from this claim. 
	
	To prove the claim, we consider the algebraic number $\kappa:=\kln_\gamma(v)\cdot q^{-\deg v/2}$.
	Up to replacing $\kln_\gamma(v)$ by $\kln'_\gamma(v)$, we can assume that $\iota(\kappa)  = \e^{i\kloan_\gamma(v)}$.
	It is relatively easy to show that 
	$[\Q(\kappa):\Q]\leq 2(p-1)$ and that $ht(\kappa)\leq \log(q^{\deg v/2})$, see \cite[Lem.\ 5.3]{Griffon_LegAS} for details.
	Let us now apply Liouville's inequality to $\kappa$ with the following three polynomials: $P_1=X-1$, $P_2=X+1$ and $P_3=X^4-1$.
	
	By \refKloos{item.kloos.avoid} or Remark \ref{rema.nonvanishing.reloaded}, we know that $\kappa\notin\{1, i, -1\}$, so that $P_j(\kappa)\neq 0$ for $j\in \{1, 2, 3\}$.
	Combined with our estimates of the degree and of the height of $\kappa$, Liouville's inequality \eqref{eq.liouville} yields that 
	\begin{align}
	\forall j \in\{1, 2, 3\}, \qquad \log |\iota(P_j(\kappa))|
	&\geq - [\Q(\kappa):\Q] \cdot (\log 2 + 4\cdot ht(\kappa)) 
	 \geq - 2(p-1)\cdot (\log 2 + 2 \log q^{\deg v})	\notag\\
	 &  \geq - 6(p-1)\cdot \log q^{\deg v}.\label{eq.liouv.inter}	
	\end{align}
	Besides, a quick analysis shows that, for any $\theta\in [0, \pi]$, one has 
	\begin{eqnarray*}
	|P_1(\e^{i\theta})|	 =  |\e^{i\theta}-1| \leq |\theta| = \theta, 
	\qquad\quad |P_2(\e^{i\theta})|	 =  |\e^{i\theta}+1| = |\e^{i(\theta-\pi)} - 1|\leq |\theta-\pi| = \pi-\theta, \\
	\text{and }|P_3(\e^{i\theta})| = |\e^{4i\theta}-1| = |\e^{i\theta}-i| \cdot |\e^{i\theta}+i| \cdot  |\e^{2i\theta}-1| \leq 4 |\e^{i\theta}-i| = 4|\e^{i(\theta-\pi/2)}-1| \leq 4|\theta-\pi/2|. 
	\end{eqnarray*}
	Applying these inequalities to $\theta=\kloan_\gamma(v)$ and using \eqref{eq.liouv.inter}, we directly obtain that 
	\[\kloan_\gamma(v) \geq |P_1(\kappa)| \geq (q^{\deg v})^{-6(p-1)}, \quad\text{and that }\quad 
	\pi - \kloan_\gamma(v) \geq |P_2(\kappa)| \geq (q^{\deg v})^{-6(p-1)}.\]
	By further noting that $q^{2\deg v} \geq 4$, we also get that
	\[ \left|{\pi}/{2} - \kloan_\gamma(v)\right| 
	\geq |P_3(\kappa)|/4 \geq (q^{\deg v})^{-6(p-1)}/4 	\geq  (q^{\deg v})^{-(6p-4)}.\]
	The claim now follows immediately, with $\sigma_p = 6p-4>0$ being a suitable choice of constant.
	\end{proof}

 %% -- %% -- %% -- %% -- %% -- %% -- %% -- %% -- %% -- %% -- %% -- %% -- %% -- %% -- %% -- %% -- %% -- %% -- %% -- %% -- %% -- %% -- %% -- %% -- %% -- %%
\section[Bounds on the central value]{Estimates of the central value $L(E_{\gamma,a}, q^{-1})$}
\label{sec.bnd.spval}

The goal of this section is to prove a precise asymptotic estimate on the central value $L(E_{\gamma,a}, q^{-1})$ as $a\to\infty$.
This will provide the crucial input for our bounds on the size of $\sha(E_{\gamma, a})$ in the next section. 

The result is as follows:
	\begin{theo}\label{theo.bounds.spval.L} 
	Let $\F_q$ be a finite field of odd characteristic, and $K=\F_q(t)$.
	There are positive constants $c_1, c_2$ (depending at most on $q$) such that: 
	for any $\gamma\in\F_q^\times$ and any $a\geq 1$, the central value $L(E_{\gamma, a}, q^{-1})$ of the $L$-function of the elliptic curve $E_{\gamma, a}$ satisfies: 
	\begin{equation}\label{eq.bounds.spval}
	- \frac{c_1}{a}
	\leq \frac{\log |L(E_{\gamma,a}, q^{-1})|}{\log H(E_{\gamma,a})} 
	\leq \frac{c_2}{a}.
	\end{equation}	
	\end{theo}

The proof of this theorem will be given in \S\ref{sec.proof.bounds.spval.L} after proving an intermediate result in the following subsection. 
The upper bound in \eqref{eq.bounds.spval} provides a slight improvement on the generic upper bound on the central value: 
for any non-isotrivial elliptic curve $E$ over $K$ with $L(E, q^{-1})\neq 0$, it is known that 
\[\frac{\log L(E, q^{-1})}{\log H(E)} \leq c_3\cdot\frac{\log\log\log H(E)}{\log\log H(E)} \qquad (\text{as } H(E)\to\infty),\]
for some constant $c_3>0$.
This can be proved by using the fact that the zeros of $L(E, T)$ become uniformly equidistributed on the circle $\{z\in\C:|z|=q^{-1}\}$ as $H(E)\to\infty$ (see \cite[Thm.\ 7.5]{HP15}).
For easier comparison, note that $\log\log H(E_{\gamma, a})$ here has the same order of magnitude as $a$, as was shown in  \eqref{eq.invariants}. 

The lower bound in \eqref{eq.bounds.spval} is, on the other hand, much stronger than the generic lower bound on the central value. 
For a non-isotrivial elliptic curve $E$ with $L(E, q^{-1})\neq 0$, the latter  only yields that
\[ -1 + \frac{c_4}{\log H(E)} \leq  \frac{\log L(E, q^{-1})}{\log H(E)}  \qquad (\text{as } H(E)\to\infty),\]
for some  $c_4<0$.
\emph{Via} the BSD formula, this translates into a lower bound on the order of $\sha(E)$ which is weaker than the one we require to prove Theorem \ref{itheo.large.sha}{\it(3)}.

\subsection{Convergence of a certain sequence}\label{sec.convergence}

Let $a\geq 1$ be an integer and $\gamma\in\F_q^\times$. 
We keep the notation introduced in section \ref{sec.distribution}.
Consider the non-negative function $W:[0,\pi]\to\R$ defined by 
\[ W(\theta):=\begin{cases}
- \log\big(\sin^2\theta\cdot\cos^2\theta\big)& \text{ if }\theta\in[0, \pi]\smallsetminus\{0, \pi/2, \pi\}, \\
0 & \text{ if } \theta\in\{0, \pi/2, \pi\}.
\end{cases}\]
The function $\theta\mapsto W(\theta)$ is continuously differentiable on $(0, \pi/2)\cup(\pi/2, \pi)$ and, for all $\theta\in[0, \pi]$, it satisfies ${W(\pi-\theta) = W(\theta)}$. 
A routine check shows that $W$ is integrable on $[0,\pi]$ for the Lebesgue measure as well as for the Sato--Tate measure $\mu_\infty$.
By Remark \ref{rema.nonvanishing.reloaded},
the function $W$ is also  integrable  for the measure~$\mu_a$ introduced in \S\ref{sec.distrib.kloos}:
it thus makes sense to consider the sequence $\left(\int_{[0,\pi]} W\dd\mu_a\right)_{a\geq1}$.
Even though we know that $\mu_a$ converges weak-$\ast$ to $\mu_\infty$ as $a\to\infty$ by Theorem \ref{theo.distrib.kloos}, we cannot directly conclude that this sequence converges because~$W$ is not continuous on $[0,\pi]$.

Nonetheless, the goal of this subsection is to show the following: 
	\begin{prop}\label{theo.convergence}
	In the above setting, 
	the sequence $\left(\int_{[0,\pi]} W\dd\mu_a\right)_{a\geq1}$ converges to $\int_{[0,\pi]}W\dd\mu_\infty$.
	More precisely, one has
	\[\left| \int_{[0,\pi]} W\dd\mu_a - \int_{[0,\pi]}W\dd\mu_\infty\right| 
	\ll_{q} \frac{a^{3/2}}{q^{a/4}} 	\qquad (\text{as }a \to\infty),\]
	where the implicit constant is effective and depends at most on $q$. 
	\end{prop}

\noindent
Even though the exact value of the limit is of little importance to us, one can actually compute that $\int_{[0, \pi]}W\dd\mu_\infty= \log(16)$.
By definition of the measure $\mu_a$, the result above can thus be rewritten as
\[ \frac{1}{|P_q(a)|} \sum_{v\in P_q(a)} \log\big(\sin^2\kloan_\gamma(v)\cdot\cos^2\kloan_\gamma(v)\big) = -\log(16) +O\left(\frac{a^{3/2}}{q^{a/4}}\right), \quad (\text{as } a\to\infty).\]

\begin{proof}
We pick, once and for all, a nondecreasing continuously differentiable function $\beta_0:[0,1]\to[0,1]$ such that 
$\beta_0(x) = 0$ for all $x\in [0, 1/3]$, and $\beta_0(x) = 1$ for all $x\in [2/3, 1]$.
For a small enough $\epsilon>0$, we   define a ``smoothing'' function $\beta_\epsilon:[0,\pi]\to[0,1]$ as follows: define $\beta_\epsilon$ on $[0, \pi/2]$ by 
\[ \beta_\epsilon(\theta) :=
\begin{cases}
\beta_0(\theta/\epsilon) & \text{if } \theta\in[0, \epsilon], \\
1 				& \text{if } \theta\in[\epsilon, \pi/2-\epsilon], \\
\beta_0((\pi/2-\theta)/\epsilon) & \text{if } \theta\in[\pi/2- \epsilon, \pi/2], 
\end{cases}\]
and extend it to $[0,\pi]$ by requiring that $\beta_\epsilon(\pi-\theta) = \beta_\epsilon(\theta)$ for all $\theta\in [0,\pi]$. 

For any $\epsilon>0$, we put $W_\epsilon := W\cdot \beta_\epsilon$.
The function $W_\epsilon : [0,\pi]\to\R$ is continuously differentiable on $[0, \pi]$, it coincides with $W$ on $[\epsilon, \pi/2-\epsilon] \cup [\pi/2+\epsilon, \pi-\epsilon]$, and it satisfies $W_\epsilon(\pi-\theta) = W_\epsilon(\theta)$ for all $\theta\in [0,\pi]$. 
Let us also record the following estimates:
\begin{lemm}\label{lemm.analytic.est}
For all $\epsilon\in (0, 1/4)$, one has
\begin{multicols}{2}
\begin{enumerate}[(a)]
\item $\displaystyle \int_0^\pi |W'_\epsilon(t)|\dd t = O\big(|\log\epsilon|\big)$,
\item $\displaystyle \int_{[0,\pi]} |W-W_\epsilon|\dd\mu_\infty  = O\big(\epsilon|\log\epsilon|\big)$.
\end{enumerate}
\end{multicols}
\end{lemm}

\noindent
We postpone proving these two bounds until the end of the section, and carry on with the proof of Proposition~\ref{theo.convergence}. 
In the notation introduced above, the triangle inequality yields that 
\begin{multline}\label{eq.Ineg.Triangle}
\left|\int_{[0,\pi]}W\dd\mu_a - \int_{[0,\pi]}W\dd\mu_\infty\right|
\leq \int_{[0,\pi]}|W-W_\epsilon|\dd\mu_a 
\\ + \left|\int_{[0,\pi]}W_\epsilon\dd\mu_a - \int_{[0,\pi]}W_\epsilon\dd\mu_\infty\right| 
+ \int_{[0,\pi]}|W-W_\epsilon|\dd\mu_\infty.
\end{multline}
We treat separately each of the three terms appearing in the right-hand side of this inequality.

The third term $\int_{[0,\pi]}|W-W_\epsilon|\dd\mu_\infty$ can be directly bounded with the help of Lemma \ref{lemm.analytic.est}{\it(b)}. 

To show that the middle term is small when $a\to\infty$, we apply Theorem \ref{theo.distrib.kloos} to $W_\epsilon$. 
This is permissible since $W_\epsilon$ is continuously differentiable on $[0,\pi]$.
We obtain that 
		 \[\left|\int_{[0,\pi]}W_\epsilon\dd\mu_a - \int_{[0,\pi]}W_\epsilon\dd\mu_\infty\right|  
		 \ll_q \frac{a^{1/2}}{q^{a/4}}\cdot \int_0^\pi |W'_\epsilon(t)|\dd t 
		 \ll_q \frac{a^{1/2}}{q^{a/4}}\cdot |\log\epsilon|, \]
where the last inequality follows from Lemma \ref{lemm.analytic.est}{\it(a)}. 

Finally, we claim that the first term $\int_{[0,\pi]}|W-W_\epsilon|\dd\mu_a$ on the right-hand side of \eqref{eq.Ineg.Triangle} vanishes provided that $\epsilon>0$ is small enough.
Indeed, we know by Theorem \ref{theo.small.angles} that the support of $\mu_a$ is contained in $[\epsilon_a, \pi/2-\epsilon_a] \cup [\pi/2 + \epsilon_a, \pi-\epsilon_a]$ where $\epsilon_a = (q^{a})^{-\sigma_p}$ for some constant $\sigma_p>0$. 
By the proof of that theorem, one can take $\sigma_p=6p-4$. 
On the other hand, as was noted in a previous paragraph, $W$ and $W_\epsilon$ coincide on $[\epsilon, \pi/2-\epsilon]\cup[\pi/2+\epsilon, \pi-\epsilon]$.
Hence, choosing $\epsilon < \epsilon_a$, we have $\int_{[0,\pi]}|W-W_\epsilon|\dd\mu_a =0$. 

Summing up these three contributions, inequality \eqref{eq.Ineg.Triangle} yields that
\[ \left|\int_{[0,\pi]}W\dd\mu_a - \int_{[0,\pi]}W\dd\mu_\infty\right| 
\ll_{q} 0 + \frac{a^{1/2}}{q^{a/4}}\cdot |\log\epsilon| + \epsilon |\log\epsilon|, \]
for all $\epsilon < \epsilon_a$. 
Note that $(q^{a})^{-6q}<\epsilon_a$ because $6q>6p-4=\sigma_p$.
We may therefore take $\epsilon = (q^{a})^{-6q}$, 
and this choice provides the desired bound. 
\end{proof}

\begin{proof}[Proof of Lemma \ref{lemm.analytic.est}] 

Both  $W$ and $W_\epsilon$ are periodic of period $\pi/2$ and are symmetric around $\pi/4$ 
(\ie{}, $W(\pi/2-\theta) = W(\theta)$ for all $\theta\in[0, \pi/2]$, and similarly for $W_\epsilon$). 
The same holds for $|W'_\epsilon|$ and $|W - W_\epsilon|$.
We are thus reduced to proving the following two bounds: 
\[ \text{\it(a')} \quad \int_{0}^{\pi/4} |W'_\epsilon(t)|\dd t = O(|\log\epsilon|) \qquad \text{ and }\qquad 
	\text{\it(b')} \quad  \int_{[0,\pi/2]} |W-W_\epsilon|\dd\mu_\infty  = O\big(\epsilon|\log\epsilon|\big). \]
To that end, we first gather a few useful remarks:
\begin{enumerate}[(i)]
\item\label{item.an.i}
Since $W_\epsilon = W\cdot \beta_\epsilon$ and since $0\leq \beta_\epsilon(t)\leq 1$,  we have 
$|W_\epsilon'(t)| \leq |W'(t)| + |\beta'_\epsilon(t)|\cdot W(t)$ for all $t\in[0,\pi]$.  

\item\label{item.an.ibis}
We have $|\beta'_\epsilon(t)|=0$ for all $t\in[\epsilon, \pi/4]$ and $|\beta'_\epsilon(t)|\leq M/\epsilon$ for $t\in[0, \epsilon]$, where $M:=\sup_{x\in[0,1]} |\beta'_0(x)|$.

\item\label{item.an.ii} 
The function  $W_\epsilon$ is constant (equal to $0$) on $[0, \epsilon/3]$, so that $W'_\epsilon(t)=0$ for all $t\in(0, \epsilon/3)$.

\item\label{item.an.iibis} 
The functions $W$ and $W_\epsilon$ are symmetric around $\pi/4$, and they coincide on $[\epsilon, \pi/2- \epsilon]$. 
Moreover, $|W(t)-W_\epsilon(t)| = (1-\beta_\epsilon(t))\cdot W(t) \leq W(t)$ for all $t\in[0,\pi/2]$. 

\item\label{item.an.iii} 
For all $t\in(0, \pi/4]$, it is classical that $(\sin t)^{-1} \leq \pi/(2t)$ and $(\cos t)^{-1} \leq \pi/(\pi-2 t)$. 
We deduce that 
\[W( t) = 2\log\big( (\sin t)^{-1} \cdot(\cos t)^{-1}\big) \leq 2\log\left(\frac{\pi^2}{2 t(\pi-2 t)}\right) \leq 2\log\frac{\pi}{ t},\]
for all $t$ in this interval. In particular, this yields that
$\int_0^\epsilon W( t)\dd t \leq 	-2 \int_0^\epsilon \log\frac{ t}{\pi}\dd t \ll \epsilon |\log\epsilon|$.

\item\label{item.an.iv}	
For all $ t\in(0, \pi/4]$, a simple calculation shows that 
$\displaystyle W'( t) = 2\left(\frac{\sin t}{\cos t} - \frac{\cos t}{\sin t}\right) = - \frac{2\cos(2 t)}{\sin t\cdot\cos t}$.
Thus, using the same classical inequalities as in the preceding item, we obtain that, for all $ t\in(0, \pi/4]$,
\[|W'( t)|\leq \frac{2}{\sin t\cdot\cos t} \leq \frac{2\pi^2}{2 t(\pi-2 t)}\leq \frac{2\pi}{t}.\]
\end{enumerate}
We can now combine the above items to prove the desired inequalities.
First, we have
\begin{align*}
\int_0^\epsilon |W'_\epsilon( t)|\dd t
	&\stackrel{\text{by \eqref{item.an.ii} }}%
{=} \int_{\epsilon/3}^\epsilon |W'_\epsilon( t)|\dd t
	\stackrel{\text{by \eqref{item.an.i} }}%
{\leq}  \int_{\epsilon/3}^\epsilon |W'( t)|\dd t + \int_{\epsilon/3}^\epsilon W( t) |\beta'_\epsilon( t)|\dd t \\
	&\stackrel{\text{by \eqref{item.an.iv}, \eqref{item.an.ibis}  }}%
{\leq} \int_{\epsilon/3}^\epsilon \frac{2\pi}{ t}\dd t + \int_{\epsilon/3}^\epsilon W( t) \cdot\frac{M}{\epsilon}\dd t 
	\stackrel{\text{by \eqref{item.an.iii}} }%
{\ll} |\log\epsilon|.
\end{align*}
Furthermore, we have
\begin{align*}
\int_\epsilon^{\pi/4} |W'_\epsilon( t)|\dd t
	\stackrel{\text{by \eqref{item.an.i} }}%
{\leq} \int_\epsilon^{\pi/4} |W'( t)|\dd t +	\int_\epsilon^{\pi/4}W( t) |\beta'_\epsilon( t)|\dd t
	\stackrel{\text{by \eqref{item.an.iv}, \eqref{item.an.ibis} }}%
{\leq} \int_\epsilon^{\pi/4}\frac{2\pi}{ t}\dd t 
\ll |\log\epsilon|.
\end{align*}
Summing the last two displays proves {\it(a')}. Finally, we notice that
\begin{align*}
\int_{[0, \pi/2]} |W-W_{\epsilon}|\dd\mu_\infty 
	& %= \frac{2}{\pi}\int_0^{\pi/2} |W(t)- W_\epsilon(t)|\sin^2(t)\dd t 
\leq \frac{2}{\pi}\int_0^{\pi/2} |W(t)- W_\epsilon(t)|\dd t
	 \stackrel{\text{by \eqref{item.an.iibis} }}%{=} \frac{4}{\pi} \int_0^\epsilon (1- \beta_\epsilon(t))W(t) \dd t 
\leq \frac{4}{\pi} \int_0^\epsilon  W(t) \dd t  
	\stackrel{\text{by \eqref{item.an.iii} }}%
{\ll} \epsilon  |\log\epsilon|, 
\end{align*}
which proves {\it(b')} and concludes the proof of the lemma.
\end{proof}

\subsection{Proof of Theorem \ref{theo.bounds.spval.L}}\label{sec.proof.bounds.spval.L}
With Proposition \ref{theo.convergence} at hand, we can now prove Theorem \ref{theo.bounds.spval.L}.
In the notation introduced in the previous sections, we begin by taking the logarithm of identity \eqref{eq.spval.expr}:
\begin{multline}\label{eq.bnd.spval.1}
\log L(E_{\gamma,a}, q^{-1})  
= \log|L(E_{\gamma,a}, q^{-1})| %= \sum_{v\in P_q(a)} \log\big|\big(1-\e^{i(\gauan(v) + \kloan_\gamma(v))}\big)\big(1-\e^{i(\gauan(v) - \kloan_\gamma(v))}\big)\big|\notag \\ &
= \sum_{v\in P_q(a)} \log\big| 1 + \e^{2i\gauan(v)}-2 \cdot \e^{i\gauan(v)}  \cdot \cos\kloan_\gamma(v) \big| \\  
= \sum_{v\in P_q(a)}  \log|F_{v}(\kloan_\gamma(v))|,
\end{multline}
where  the functions $F_{v}:[0,\pi]\to\R$  are defined as follows:
\[\forall v\in P_q(a), \qquad 
F_v: \theta\mapsto
\begin{cases}
 2-2\cos\theta & \text{if } \gauan(v)= 0, \\
-2i\cos\theta & \text{if } \gauan(v)= \pi/2, \\
2+2\cos\theta & \text{if } \gauan(v)= \pi, \\
2i\cos\theta & \text{if } \gauan(v)= 3\pi/2. \\
\end{cases}\]
Note that  $|F_v(\kloan_\gamma(v))|>0$ for all $v\in P_q(a)$ because $\kloan_\gamma(v)\notin\{0, \pi/2, \pi\}$ by \refKloos{item.kloos.avoid}  (see Remark \ref{rema.nonvanishing.reloaded}). 
Straightforward analytic estimates show that, for any place $v\in P_q(a)$, one has
 \[\forall\theta\in [0,\pi], \qquad 
 \sin^2\theta\cdot\cos^2\theta \leq |F_{v}(\theta)|\leq 4.\]
In particular, writing $W(\theta) = - \log\big(\sin^2\theta\cdot\cos^2\theta\big)$ as in \S\ref{sec.convergence}, we obtain that 
\[\forall v\in P_q(a), \qquad - W(\kloan_\gamma(v)) \leq \log |F_v(\kloan_\gamma(v))| \leq \log 4.\]
Summing this chain of inequalities over all $v\in P_q(a)$, equality \eqref{eq.bnd.spval.1} leads to
\[ \frac{1}{\log H(E_{\gamma, a})}\cdot\sum_{v\in P_q(a)} -W(\kloan_\gamma(v)) 
\leq \frac{\log L(E_{\gamma, a},q^{-1})}{\log H(E_{\gamma, a})} 
\leq \log 4\cdot \frac{|P_q(a)|}{\log H(E_{\gamma, a})}.\]
It is clear from these inequalities that Theorem \ref{theo.bounds.spval.L} will be proved once we show the following two assertions: 
\begin{enumerate}[$(C_1)$]
    \item 
    	There exists a constant $c_1>0$ such that 
    	$\displaystyle \frac{1}{\log H(E_{\gamma, a})} \sum_{v\in P_q(a)} W(\kloan_\gamma(v)) \leq c_1 \cdot a^{-1}$,
    \item 
    	There exists a constant $c_2>0$ such that 
    	$\displaystyle \frac{|P_q(a)|}{\log H(E_{\gamma, a})}\leq c_2 \cdot a^{-1}$.
\end{enumerate}
We thus now prove these two claims, starting with the second one.

We know from \eqref{eq.invariants} that $\log H(E_{\gamma, a}) \gg_q q^{a} $, and from \eqref{eq.estimate.Pqa} that  $|P_q(a)|\ll_q q^a/a$. 
The second claim $(C_2)$ follows from these two estimates.
Next we turn to the proof of $(C_1)$: by Proposition \ref{theo.convergence}, we have
\[\frac{1}{|P_q(a)|}\sum_{v\in P_q(a)} W(\kloan_\gamma(v)) 
= \int_{[0,\pi]}W\dd\mu_\infty + O_{q}\left(\frac{a^{3/2}}{q^{a/4}}\right),\]
as $a\to\infty$, where the implicit constant depends at most on $q$. 
Therefore, making use of the estimate in $(C_2)$ which we have just proved, we obtain that
\begin{align*}
0 \leq \frac{1}{\log H(E_{\gamma, a})}\sum_{v\in P_q(a)} W(\kloan_\gamma(v))
&=\frac{|P_q(a)|}{\log H(E_{\gamma, a})} \cdot \frac{1}{|P_q(a)|}\sum_{v\in P_q(a)} W(\kloan_\gamma(v)) \\
&\leq \frac{c_2}{a} \cdot \left(\int_{[0,\pi]}W\dd\mu_\infty + O_{q}\left(\frac{a^{3/2}}{q^{a/4}}\right)\right)
\leq \frac{c_1}{a},
\end{align*}
for some constant $c_1>0$, depending at most on  $q$. 
This proves the first claim $(C_1)$ and concludes the proof of Theorem \ref{theo.bounds.spval.L}.
\hfill$\Box$

 %% -- %% -- %% -- %% -- %% -- %% -- %% -- %% -- %% -- %% -- %% -- %% -- %% -- %% -- %% -- %% -- %% -- %% -- %% -- %% -- %% -- %% -- %% -- %% -- %% -- %%
\section{Proof of Theorem \ref{itheo.large.sha}}

In this section, we gather our results so far to prove our main result (Theorem \ref{itheo.large.sha} in the introduction).
We have already proved assertions {\it (1)} and {\it(2)} of that theorem: see Proposition \ref{prop.noniso} and Corollary \ref{coro.BSD}{\it(2)}, respectively.
In the following two subsections, we prove assertions {\it(4)} and {\it (3)} of Theorem \ref{itheo.large.sha}, in this order.

\subsection[p-part of sha]{The $p$-primary part of $\sha(E_{\gamma, a})$}
\label{sec.ppart.sha}

Let us first prove  assertion {\it(4)} of Theorem \ref{itheo.large.sha} concerning the $p$-primary part of $\sha(E_{\gamma, a})$. That is to say,

\begin{theo}\label{theo.ppart.sha}
Let $\F_q$ be a finite field of odd characteristic $p$, and $K=\F_q(t)$.
 For any $\gamma\in\F_q^\times$ and any integer $a\geq 1$, consider the elliptic curve $E_{\gamma, a}$ defined over $K$ by \eqref{eq.Wmod}.
Then the $p$-primary part of $\sha(E_{\gamma,a})$ is trivial. 
In other words, the integer $|\sha(E_{\gamma, a})|$ is relatively prime to $p$.
\end{theo}

	\begin{proof}
	Recall from \S\ref{sec.padic.slopes} that $\ord_\gP:\Qbar^\times\to\Q$ denotes an extension of the $p$-adic value to $\Qbar$, normalised so that $\ord_\gP(q)=1$.
	Taking $\gP$-adic valuation of both sides of the BSD formula \eqref{eq.BSD.2}, we obtain that
	\begin{equation}\label{eq.ppart.sha}
	\ord_\gP|\sha(E_{\gamma, a})| 
	= \ord_\gP L(E_{\gamma, a}, q^{-1}) +\ord_\gP H(E_{\gamma, a}) -1.	
	\end{equation}
	Our formula \eqref{eq.invariants} for the height of $E_{\gamma, a}$ implies that $ \ord_\gP H(E_{\gamma, a}) -1 = {(q^a-1)}/{2}$. 
	To conclude the proof, it therefore suffices to prove that $\ord_\gP L(E_{\gamma, a}, q^{-1}) = -(q^a-1)/2$. 
	Indeed one would then have $\ord_\gP|\sha(E_{\gamma, a})|=0$, and \eqref{eq.ppart.sha} would directly show that $\sha(E_{\gamma, a})$ has trivial $p$-primary part (by the structure theorem for finite abelian groups).
	
	We now proceed to compute $\ord_\gP L(E_{\gamma, a}, q^{-1})$. 
	Evaluating at $T=q^{-1}$ the expression for $L(E_{\gamma, a}, T)$ obtained in Theorem \ref{theo.Lfunc}, and taking $\gP$-adic valuations on both sides of the resulting identity yields that
	\[ \ord_\gP L(E_{\gamma, a}, q^{-1}) 
	= \sum_{v\in P_q(a)} \ord_\gP\left(1-\frac{\gaun(v)\kln_\gamma(v)}{q^{\deg v}}\right) + \ord_\gP\left(1-\frac{\gaun(v)\kln'_\gamma(v)}{q^{\deg v}}\right).\]
	The results proved in \S\ref{sec.padic.slopes} imply that, for all places $v\in P_q(a)$, we have 
	\[  \left\{\ord_\gP\left({\gaun(v)\kln_\gamma(v)}{q^{-\deg v}}\right),  \ord_\gP\left({\gaun(v)\kln'_\gamma(v)}{q^{-\deg v}}\right)\right\}
	=\left\{-\frac{\deg v}{2}, \frac{\deg v}{2}\right\}.\]
	Using the cases of equality in the non-archimedean triangle inequality, we then obtain that
	\[ \ord_\gP L(E_{\gamma, a}, q^{-1}) 
	= \sum_{v\in P_q(a)} \min\left\{0, -\frac{\deg v}{2} \right\} + \min\left\{0,  \frac{\deg v}{2} \right\} 
	= -\frac{1}{2}\sum_{v\in P_q(a)} \deg v.\]
	As has already been observed, we have $\sum_{v\in P_q(a)} \deg v = |\mathbb{G}_m(\F_{q^a})| =q^a-1$.
	We therefore conclude that $\ord_\gP L(E_{\gamma, a}, q^{-1}) = -(q^a-1)/2$, as was to be shown.
	\end{proof}

	\begin{rema}
	\begin{enumerate}[(1)]
	\item	
	In a very recent paper \cite{Ulmer_BSalg}, Ulmer introduces the notion of  \emph{dimension of $\sha$} for abelian varieties over $\F_q(t)$ whose Tate--Shafarevich group is finite.
	In the special case of $E_{\gamma, a}$, this invariant is defined as follows:
	for all integers $n\geq 1$, let $K_n:=\F_{q^n}(t)$ and consider 
	\[\dim\sha(E_{\gamma, a}):=\lim_{n\to\infty} \frac{\log\big|\sha\big(E_{\gamma, a}\times_{K}K_n \big/ K_n\big)[p^\infty]\big|}{\log q^n}.\]
	Proposition 4.1 of \cite{Ulmer_BSalg} shows that the limit indeed exists and is a nonnegative integer: it is called the  \emph{dimension of $\sha$  of $E_{\gamma, a}$}.
	Furthermore, \cite[Prop.\ 4.2]{Ulmer_BSalg} provides an expression for~$\dim\sha(E_{\gamma, a})$ in terms of  the $p$-adic slopes of the $L$-function of $E_{\gamma, a}$. 
	Using that expression and our Theorem \ref{theo.Lfunc.slopes}, an easy calculation yields that $\dim\sha(E_{\gamma,a}) = 0$.
	\item 
	The previous item show that %the sequence $n\mapsto \log \big|\sha(E_{\gamma, a}\times_K K_n/K_n)[p^\infty]\big|$ is
	the order of the $p$-primary part of the Tate--Shafarevich group of the base-changed elliptic curve $E_{\gamma, a}\times_K K_n$ over $K_n$ grows slowly with $n$
	(its $\log$ is $o(n)$). % as $n\geq 1$ grows. 
	By replacing~$q$ by $q^n$ in the proof of Theorem \ref{theo.ppart.sha}, one can actually prove the stronger statement that, for all $n\geq 1$, 
	\[|\sha(E_{\gamma, a}\times_K K_n/K_n)[p^\infty]|=1.\] 
	\end{enumerate}
	\end{rema}

\subsection[Bounds on the Tate-Shafarevich group]{The size of $\sha(E_{\gamma, a})$}
\label{sec.large.sha}

Finally, we prove  assertion {\it(3)} of Theorem~\ref{itheo.large.sha} about the size of $\sha(E_{\gamma, a})$.
We actually show a slightly more precise result: 

	\begin{theo}\label{theo.bounds.sha.H}
	 Let $\F_q$ be a finite field of odd characteristic and $K=\F_q(t)$.
	 For any  $\gamma\in\F_q^\times$ and any integer $a\geq 1$,  consider the elliptic curve $E_{\gamma, a}$ defined over $K$ by \eqref{eq.Wmod}.
	 Then, as $a\to\infty$, we have
	 \[ |\sha(E_{\gamma, a})| 
	 = H(E_{\gamma, a})^{1+O(1/a)},\]
	 where the implicit constant is effective and depends at most on $q$.
	\end{theo}

\noindent
We know  by \eqref{eq.invariants} that $H(E_{\gamma, a}) = N(E_{\gamma, a})^{1/4}$. 
Hence, we deduce from the above that 

	\begin{coro}\label{coro.bounds.sha.N}
	In the same setting,  as $a\to\infty$, we have
	 \[ |\sha(E_{\gamma, a})| 
	 = N(E_{\gamma, a})^{1/4+O(1/a)}.\]
	\end{coro}
	
	\begin{proof}[Proof (of Theorem \ref{theo.bounds.sha.H})]
	We note that, by \eqref{eq.invariants}, $\log \log H(E_{\gamma, a})$ and $a$ have the same order of magnitude when $a\to\infty$, the involved constants depending at most on $q$.
	The elliptic curve $E_{\gamma, a}$ satisfies the BSD conjecture (see Corollary \ref{coro.BSD}).
	Taking the logarithm of both sides of the BSD formula \eqref{eq.BSD.2} and reordering terms yields that
	\begin{equation}\label{eq.bounds.inter}
	\frac{\log|\sha(E_{\gamma, a})|}{\log H(E_{\gamma, a})} =  1 - \frac{\log q}{\log H(E_{\gamma, a})} + \frac{\log L(E_{\gamma, a}, q^{-1})}{\log H(E_{\gamma, a})}.
	\end{equation}
	The term ${\log q}/{\log H(E_{\gamma, a})}$ is clearly $o\big(1/\log\log H(E_{\gamma, a})\big)$ as $a\to\infty$. 
	To control the right-most term, we put to use our bound on the central value $ L(E_{\gamma, a}, q^{-1})$: we deduce from Theorem \ref{theo.bounds.spval.L} that
	\[ \frac{\big|\log L(E_{\gamma, a}, q^{-1})\big|}{\log H(E_{\gamma, a})} = O\big( 1/\log\log H(E_{\gamma, a})\big), \]
	as $a\to\infty$, where the implicit constant depends at most on $q$.
	The proof is now completed by plugging these two estimates into \eqref{eq.bounds.inter}.
	\end{proof}
	
	\begin{rema}
	Theorem \ref{theo.bounds.sha.H} can be interpreted as an analogue of the Brauer--Siegel theorem for the sequences~$\{E_{\gamma, a}\}_{a\geq 1}$.
	We refer the reader to \cite{HP15, Hindry_BSinAG} for a detailed description of the analogue we have in mind: let us simply recall that Hindry and Pacheco have introduced the \emph{Brauer--Siegel ratio} of an abelian variety with finite Tate--Shafarevich group and that,
	in the case at hand, the Brauer--Siegel ratio  is given by
	$\BS(E_{\gamma, a})  = {\log|\sha(E_{\gamma, a})|}\big/{\log H(E_{\gamma, a})}$.
	For an elliptic curve with positive Mordell--Weil rank, the Brauer--Siegel ratio also includes the N\'eron--Tate regulator.
	With this notation, Theorem \ref{theo.bounds.sha.H} can be rewritten in a compact form:
	\begin{equation}\label{eq.BS.lim}
	\BS(E_{\gamma,a}) =1 + O\big(1/a\big) 
	\qquad(\text{as } a\to\infty).	
	\end{equation}	
	There are only a handful of sequences of elliptic curves over $K$ for which one can prove that the Brauer--Siegel ratio has a limit and that this limit is $1$ (see \cite{HP15, Griffon_PHD, Griffon_LegAS, Hindry_BSinAG} and references therein). 
	The families studied in the present paper therefore provide further examples of that behaviour. 
	However, it seems interesting to remark that the sequences $\{E_{\gamma, a}\}_{a\geq1}$ consist of rank $0$ elliptic curves, % for which the Brauer--Siegel ratio is known to have limit $1$.
	whereas previous articles considered sequences with unbounded rank.
	 \end{rema}

%% -- END CONTENT  -- %%  
\noindent\hfill\rule{7cm}{0.5pt}\hfill\phantom{.}

%% -- THANKS  -- %% 
\paragraph{Acknowledgements} 
This article is partly based on GdW's Leiden Master's thesis \cite{deWit_thesis}, written under the supervision of RG.
The authors would like to thank Bas Edixhoven, Marc Hindry and Douglas Ulmer for fruitful conversations about this work and useful comments on earlier versions thereof.

GdW received funding from the ALGANT Master Programme during his studies at Universit\`a degli studi di Milano and Universiteit Leiden.
Work on this project was started at Universiteit Leiden, and RG now works at Universit\"at Basel. 
These institutions are gratefully thanked for creating great working conditions, and for their financial support.  
RG also acknowledges funding from Agence Nationale de la Recherche (Grant ANR-17-CE40-0012 FLAIR).

%% -- ADDRESSES -- %% 
\normalsize
\vfill

\noindent\rule{7cm}{0pt}

{Richard Griffon} {(\it \href{richard.griffon@unibas.ch}{richard.griffon@unibas.ch})}  --
{\sc Departement Mathematik, Universit\"at Basel,} 
 
Spiegelgasse 1, 4051 Basel (Switzerland). 

\medskip
{Guus de Wit} {(\it \href{mailto:dewit.guus@gmail.com}{{dewit.guus@gmail.com}})} --
{\sc Mathematisch Instituut, Universiteit Leiden,} 
 
PO Box 9512, 2300 RA Leiden (the Netherlands).

 %% -- END DOC -- %%

\small
\begin{thebibliography}{MW94}
\bibitem[GS95]{GoldfeldSzpiro}
Dorian Goldfeld and Lucien Szpiro.
\newblock Bounds for the order of the {T}ate-{S}hafarevich group.
\newblock {\em Compositio Math.}, 97(1-2):71--87, 1995.
%\newblock Special issue in honour of Frans Oort.

\bibitem[Gri16]{Griffon_PHD}
Richard Griffon.
\newblock {\em {A}nalogues du th{\'e}or{\`e}me de {B}rauer-{S}iegel pour
  quelques familles de courbes elliptiques}.
\newblock PhD thesis, Universit{\'e} Paris Diderot, July 2016.
\newblock (available at
  \href{http://math.richardgriffon.me/thesis/Griffon_thesis.pdf}{math.richardgriffon.me/thesis/Griffon\_thesis.pdf}).

\bibitem[Gri18]{Griffon_LegAS}
Richard Griffon.
\newblock Bounds on special values of ${L}$-functions of elliptic curves in an
  {A}rtin-{S}chreier family.
\newblock {\em European Journal of Mathematics}, 5(2):476--517, 2018.

\bibitem[GU19]{GriffonUlmer}
Richard Griffon and Douglas Ulmer.
\newblock On the arithmetic of a family of twisted constant elliptic curves.
\newblock (Preprint \href{https://arxiv.org/abs/1903.03901}{ArXiv:1903.03901}),
  March 2019.
  
\bibitem[Gro11]{Gross_bsd}
Benedict~H. Gross.
\newblock Lectures on the conjecture of {B}irch and {S}winnerton-{D}yer.
\newblock In {\em Arithmetic of {$L$}-functions}, vol.\ 18 of {\em IAS/Park
  City Math. Ser.}, pp. 169--209. Amer. Math. Soc., Providence, RI, 2011.

\bibitem[HP16]{HP15}
Marc Hindry and Am{\`\i}lcar Pacheco.
\newblock An analogue of the {B}rauer--{S}iegel theorem for abelian varieties
  in positive characteristic.
\newblock {\em Moscow Math. J.}, 16(1):45--93, January--March 2016.

\bibitem[Hin19]{Hindry_BSinAG}
Marc Hindry.
\newblock Analogues of {B}rauer-{S}iegel theorem in arithmetic geometry.
\newblock In {\em Arithmetic geometry: computation and applications}, vol.\
  722 of {\em Contemp. Math.}, pp.\ 19--41. Amer. Math. Soc., Providence, RI,
  2019.
  
\bibitem[LN97]{LidlN}
Rudolf Lidl and Harald Niederreiter.
\newblock {\em Finite fields}, volume~20 of {\em Encyclopedia of Mathematics
  and its Applications}.
\newblock Cambridge University Press, second edition, 1997.
\newblock With a foreword by P. M. Cohn.

\bibitem[MM94]{MaiMurty_QuadTwists}
Liem {Mai} and M.~Ram {Murty}.
\newblock {A note on quadratic twists of an elliptic curve}.
\newblock In {\em {Elliptic curves and related topics}}, pp.\ 121--124.
  American Mathematical Society, 1994.

\bibitem[MW94]{MiWa}
Maurice Mignotte and Michel Waldschmidt.
\newblock On algebraic numbers of small height: linear forms in one logarithm.
\newblock {\em J. Number Theory}, 47(1):43--62, 1994.

\bibitem[Mil75]{Milne_conjAT}
James~S. Milne.
\newblock On a conjecture of {A}rtin and {T}ate.
\newblock {\em Ann. of Math. (2)}, 102(3):517--533, 1975.

\bibitem[PU16]{UlmerPries}
Rachel Pries and Douglas Ulmer.
\newblock Arithmetic of abelian varieties in {A}rtin-{S}chreier extensions.
\newblock {\em Trans. Amer. Math. Soc.}, 368(12):8553--8595, 2016.

\bibitem[Ros02]{Rosen}
Michael Rosen.
\newblock {\em Number theory in function fields}, volume 210 of {\em Graduate
  Texts in Mathematics}.
\newblock Springer-Verlag, New York, 2002.

\bibitem[SS10]{SchShio}
Matthias Sch{{\"u}}tt and Tetsuji Shioda.
\newblock Elliptic surfaces.
\newblock In {\em Algebraic geometry in {E}ast {A}sia---{S}eoul 2008},
  vol.\ 60 of {\em Adv. Stud. Pure Math.}, pp.\ 51--160. Math. Soc. Japan,
  Tokyo, 2010.
  
\bibitem[Sil94]{ATAEC}
Joseph~H. Silverman.
\newblock {\em Advanced topics in the arithmetic of elliptic curves}, volume
  151 of {\em Graduate Texts in Mathematics}.
\newblock Springer-Verlag, New York, 1994.

\bibitem[Sil09]{AEC}
Joseph~H. Silverman.
\newblock {\em The arithmetic of elliptic curves}, volume 106 of {\em Graduate
  Texts in Mathematics}.
\newblock Springer, Dordrecht, 2nd edition, 2009.

\bibitem[Tat66]{Tate_BSD}
John~T. Tate.
\newblock On the conjectures of {B}irch and {S}winnerton-{D}yer and a geometric
  analog.
\newblock In {\em S{\'e}minaire {B}ourbaki, {V}ol.\ 9}, pages 415--440 (Exp.\
  No.\ 306). Soc. Math. France, Paris, 1965/66.

\bibitem[Ulm11]{UlmerParkCity}
Douglas Ulmer.
\newblock Elliptic curves over function fields.
\newblock In {\em Arithmetic of {$L$}-functions}, vol.\ 18 of {\em IAS/Park
  City Math. Ser.}, pp.\ 211--280. Amer. Math. Soc., Providence, RI, 2011.

\bibitem[Ulm19]{Ulmer_BSalg}
Douglas Ulmer.
\newblock {O}n the {B}rauer-{S}iegel ratio for abelian varieties over function
  fields.
\newblock {\em Algebra \& Number Theory}, 13(5):1069--1120, 2019.

\bibitem[dW98]{deWeger_ABC}
Benjamin M.~M. de~Weger.
\newblock {$A+B=C$} and big {$\sha$'s}.
\newblock {\em Quart. J. Math. Oxford Ser. (2)}, 49(193):105--128, 1998.

\bibitem[dW18]{deWit_thesis}
Guus de~Wit.
\newblock {\em Elliptic curves over function fields with large {T}ate--{S}hafarevich groups}. 
\newblock Master's thesis, Universiteit Leiden, July 2018.
\newblock (available at
\href{https://www.math.u-bordeaux.fr/~ybilu/algant/documents/theses/de_Wit.pdf}{www.math.u-bordeaux.fr/\~{}ybilu/algant/algant\_theses.php}). 
%{www.math.u-bordeaux.fr/\~{}ybilu/algant/documents/theses/de\_Wit.pdf}).


\end{thebibliography}
\end{document}